\documentclass[a4paper,11pt]{article}

\usepackage{
amsmath,
amsthm,
amssymb,
amsfonts,
enumerate,
verbatim,
graphicx,
mathtools,
color,
url,
listings,
hyperref
}

\numberwithin{equation}{section}

\def\bbb{\boldsymbol}
\def\mmm{\mathcal}
\def\fff{\mathfrak}
\def\ttt{\tilde}
\def\R{\mathbb{R}}
\def\mbR{\mathbb{R}}

\def\N{\mathbb{N}} 
\def\eps{\varepsilon}
\def\surf{{\Gamma^{0}}}
\def\surfh{{\Gamma^{0}_{h}}}
\def\solu{\bbb{u}}
\def\solw{\bbb{w}}
\def\soluh{\bbb{U}_{h}}
\def\solwh{\bbb{W}_{h}}
\def\soluhl{\bbb{u}_{h}}
\def\solwhl{\bbb{w}_{h}}
\def\proju{\bbb{\rho}^{(\solu)}}
\def\projw{\bbb{\rho}^{(\solw)}}
\def\erru{\bbb{\theta}^{(\solu)}}
\def\errw{\bbb{\theta}^{(\solw)}}
\def\snab{\nabla_{\Gamma^{0}}}
\def\snabh{\nabla_{\Gamma^{0}_{h}}}
\def\unitnormal{\bbb{\nu}_{\surf}}
\def\unitnormalh{\bbb{\nu}_{\surfh}}
\def\tfu{\bbb{\phi}}
\def\tfw{\bbb{\eta}}
\def\tfuh{\bbb{\Phi}_{h}}
\def\tfwh{\bbb{H}_{h}}
\def\tfuhl{\bbb{\phi}_{h}}
\def\tfwhl{\bbb{\eta}_{h}}

\def\pd{\partial}

\newtheorem{theorem}{Theorem}[section]
\newtheorem{problem}[theorem]{Problem}
\newtheorem{lemma}[theorem]{Lemma}
\newtheorem{corollary}[theorem]{Corollary}
\DeclareMathOperator*{\esssup}{ess\,sup}
\DeclareMathOperator*{\trace}{trace}
\DeclareMathOperator*{\diam}{diam}
\DeclareMathOperator*{\sign}{sign}

\title{A surface finite element method for computational modelling of cell blebbing}

\author{
Bj\"orn Stinner 
\and 
Andreas Dedner 
\and 
Adam Nixon
\thanks{
Corresponding author: Bj\"orn Stinner, 
Mathematics Institute, 
Zeeman Building, 
University of Warwick, 
Coventry CV4 7AL, 
United Kingdom 
\url{bjorn.stinner@warwick.ac.uk}
}
}

\begin{document}
\lstset{language=Python, basicstyle=\small}

\maketitle

\begin{abstract}

Cell blebs are protrusions of the cell membrane and can be instrumental for cell migration. We derive a continuum model for the mechanical and geometrical aspects of the onset of blebbing in terms of a force balance. It is abstract and flexible in that it allows for amending force contributions related to membrane tension or the presence of linker molecules between membrane and cell cortex. The deforming membrane and all forces are expressed by means of a parametrisation over a stationary reference surface. A variational formulation is presented and analysed for well-posedness. For this purpose, we derive a semi-discrete scheme based on the surface finite element method. We provide a convergence result and estimates of the error due to the spatial discretisation. Furthermore, we present a computational framework where specific models can be implemented and later on conveniently amended if desired, using a domain specific language implemented in Python. While the high level program control can be done within the Python scripting environment, the actual computationally expensive step of evolving the solution over time is carried out by binding to an efficient software backend. Cell membrane geometries given in terms of a parametrisation or obtained from image data can be accounted for. A couple of numerical simulation results illustrate the approach. 

\end{abstract}

\bigskip


\bigskip

\noindent \textbf{MSC(2010):}
% Primary 
65M60, %Finite elements, Rayleigh-Ritz and Galerkin methods, finite methods 
% Secondary 
92C10, %Biomechanics [See also 74L15] 
74K15, %Membranes 
92C17, %Cell movement (chemotaxis, etc.) 
74S05 %Finite element methods (Mechanics of deformable solids) 

%%%%%%%%%%%%%%%%%%%%%%%%%%%%%%%%%%%%%%%%%%%%%%%%%%%%%%%%%%%%%%%
\section{Introduction}

Cell blebbing refers to the detachment of the plasma membrane from its actin cytoskeleton and the fast formation of a spherical protrusion. This is then followed by a slower reformation of the actin cortex close to the deformed part of the membrane and a retraction phase \cite{FacGro_CellBiol_2008,ChaPal_Nature_2008}. The phenomenon is observed in various processes including apoptosis, spreading, migration, division, embryonic development, and viral entry (see \cite{Pon_rev_2017} for a recent overview), whence it is the subject of significant ongoing research. 

Mathematical models that serve to provide some insight into the control mechanisms require an approach to describe the evolving geometry and have to account for various force contributions acting on the plasma membrane. 
It seems consensus in the literature that pressurised cytosol by actin-myosin activity in the cortex or otherwise (see \cite{ChaEtAl_Nature_2005} and references) triggers blebbing, where the pressure distribution and the blebbing dynamics are of particular interest \cite{StrGuy_BiophysJ_2016,FanHuiWeiEtAl_SciRep_2017}. 
Models for the cell membrane, which resist bending, are often based on models for biomembranes or elastic shells \cite{Hel_ZNaturfC_1973,WooEtAl_Biomech_2014}, see \cite{WooEtAl_IMAJAM_2014} for a discussion of minimal approaches. 
Tension is usually accounted for, too, \cite{SchLieKerKoz_BiophysJ_2014} but significant stretching of biomembranes beyond a few percentages leads to rupture, whence the provision of cell membrane area to allow for increases as observed during blebbing has been studied \cite{GouTarMilRazEtAl_DevelCell_2017}. 
Linker molecules serve to keep the cell membrane close to the cortex but break during blebbing \cite{StrGuy_MathMedBiol_2013}, where the cell's control ability by its biochemistry is of interest \cite{WerBurPie_2019}. 
Further questions are related to the interaction of different modes during cell motility (pseudopods versus blebs, \cite{TysZatBreKay_PNAS_2014}), whether cortex weakening is required \cite{ColEtAl_Nature_2017}, or about the origin of the fluid in the bleb (from outside of the cell via pores in the membrane \cite{TalKarSalTruZapLaP_PRL_2015} or from inside through the cortex \cite{GouBoqOliRaz_PlosONE_2019}). 

Most computational methods for cell blebbing are restricted to two spatial dimensions (2D). The membrane can then be tracked by a closed polygonal chain. Forces due to its elastic properties can be computed using finite difference techniques \cite{TysZatBreKay_PNAS_2014,ColEtAl_Nature_2017}. 
If (viscous) fluid flow inside and outside of the cell is accounted for then these membrane forces can be incorporated into the flow equations by smoothing the surface distribution (immersed boundary method). For instance, in \cite{YouMit_Biomech_2010} a vorticity-stream formulation for the flow is used, and in \cite{StrGuy_MathMedBiol_2013} a staggered grid finite difference method. As an alternative there are boundary element formulations that are set up directly on the polygonal chain \cite{LimKoonChiam_ComputMethBiomechBiomedEng_2013}. 
Only very recently, results on simulations in three spatial dimensions (3D) that account for the computationally expensive flow have been published. For instance, in \cite{CamBag_PhysFluids_2017,CamBag_SoftMatter_2018} a surface finite element method for reaction-diffusion equations on the cell membrane is coupled with a projection method on a regular mesh for the flow. Alternatively to tracking the membrane with a mesh, interface capturing methods may be used. We are not aware of such an approach to cell blebbing but in \cite{MouGom_ComputMethApplMechEng_2017} numerical simulations of a phase field model for moving cells are presented, which uses isogeometric analysis for the spatial approximation and a second order stable time discretisation involving a two-stage predictor-corrector scheme. 

Our general objective has been to develop a robust and efficient computational framework for assessing and validating blebbing models in three spatial dimensions. Towards this general objective, we specifically address the following findings in this work: 

\begin{itemize}
 \item 
 The nucleation and formation of blebs is studied using a parametric approach for the moving membrane and surface finite elements that are set up on the initial membrane. 
 
 \item 
 The force balance model in \cite{TysZatBreKay_PNAS_2014,ColEtAl_Nature_2017}, which is based on ideas in \cite{YouMit_Biomech_2010,StrGuy_MathMedBiol_2013}, is extended from curves in 2D to surfaces in 3D. More precisely, a continuum model is presented such that, when restricting the model to a curve in 2D and discretising the governing equations using standard finite difference methods, the original computational model is obtained. This {\it specific model} has been used for some numerical simulations. 
 
 \item 
 In addition to this specific model we also present and analyse a {\it variational problem} that is {\it abstract} in the sense that force contributions can be altered within some structural limits. We here particularly have the models for the membrane tension and the linker molecules connecting the membrane with the cortex in mind. 
 
 \item 
 The abstract problem is analysed for well-posedness, which is based on a Galerkin method using surface finite element techniques. Stability estimates for the semi-discrete scheme are derived and exploited to show convergence. Under slightly more restrictive assumptions on the quality of the solution, error estimates have been derived, too. 
 
 \item 
 Furthermore, a software framework for numerical simulations has been developed. It features a high-level interface to implement a problem in the {\em Unified Form Language} (UFL) \cite{AlnEtAl_ACMDL_2014}, which enables to conveniently alter the variational problem. Whilst the overall program control and the time stepping are done at the high level, bindings to the {\em Distributed Unified Numerics Environment} (DUNE) \cite{BastianEtAl_DUNE_1,BastianEtAl_DUNE_2} are used for efficiently discretising and solving the spatial problems, more precisely, the Python bindings to the DUNE-FEM module \cite{DedKloNolOhl_Comp_2010,DedNol_2018_prep}. 
 
\end{itemize}

In the following Section \ref{sec:blebs} we model the onset of cell-blebbing and introduce the abstract variational problem. The finite element approach is presented and analysed in Section \ref{sec:sfem}, where we discuss well-posedness of the variational problem and convergence. Regarding the software framework, Section \ref{sec:software} contains the time discretisation, details on the implementation, and some results of our numerical simulations.

%%%%%%%%%%%%%%%%%%%%%%%%%%%%%%%%%%%%%%%%%%%%%%%%%%%%%%%%%%%%%%%
\section{Continuum modelling of the onset of blebbing}
\label{sec:blebs}

\subsection{Setting and notation}

The blebbing cell occupies an open, time dependent bounded domain denoted by $\Omega(t)$, where $t\in[0,T]$ with some $T>0$ stands for time. Its evolving boundary $\Gamma(t) = \partial \Omega(t)$ describes the position of the cell membrane and is parametrised over the initial (smooth) surface $\surf = \Gamma(0)$, i.e. $\Gamma(t)=\solu(\surf,t)$ for some function $\solu:\surf\times[0,T]\rightarrow\R^3$ such that $\solu(\cdot,0) = \bbb{id}_{\surf}$ is the identic map of $\surf$. The dependence on $t$ will usually be dropped in the following. We denote by $d$ the signed distance to $\surf$, which is well-defined in a thin layer around $\surf$, with the convention that $d<0$ inside of $\Omega(0)$. Its derivative $\unitnormal = \nabla d$ is then the outwards pointing unit normal and its second derivative $\mmm{H} = \nabla^2 d$ is the shape operator of $\surf$. By $\bbb{\kappa} = \trace(\mmm{H}) \unitnormal$ we denote the curvature vector. The surface gradient is defined by 
\begin{equation} \label{eq:def_snab}
 \snab \eta = \nabla \eta - \unitnormal (\unitnormal \cdot \nabla \eta) = \bbb{P} \nabla \eta
\end{equation}
for any differentiable function $\eta : \surf \to \R$ extended to a thin layer around $\surf$. Here, $\bbb{P} = \bbb{I} - \unitnormal \otimes \unitnormal = \snab \bbb{id}_{\surf}$, with the identity matrix $\bbb{I} \in \mbR^{3 \times 3}$, is the projection to the tangent space. The Laplace-Beltrami operator on $\surf$ is denoted and given by $\Delta_{\surf} = \snab \cdot \snab$. When integrating over $\surf$ we write $d\sigma$ for the surface area element.

\subsection{Force balance and strong formulation}

Based on previous ideas \cite{YouMit_Biomech_2010,StrGuy_MathMedBiol_2013,TysZatBreKay_PNAS_2014,ColEtAl_Nature_2017} we postulate that a force balance of the form 
\begin{equation} \label{eq:forcebal}
 \bbb{f}_{pressure} + \bbb{f}_{coupling} + \bbb{f}_{tension} + \bbb{f}_{reg} + \bbb{f}_{drag} = 0
\end{equation}
governs the cell membrane's shape. The contributions are described below. We {\it specifically} aim for generalizing the model in \cite{TysZatBreKay_PNAS_2014,ColEtAl_Nature_2017} for curves in 2D to surfaces in 3D. With regards to the coupling and tension contributions $\bbb{f}_{coupling}$, $\bbb{f}_{tension}$ we also state {\it abstract, general} forms where only some structural assumptions are made. This way, other models for these force contributions can be investigated within our framework. 

\begin{itemize}
 \item {\em Pressure}: Building up internal pressure by actin-myosin contraction is essential to blebbing. We write the corresponding force as 
 \begin{equation} \label{eq:force_pressure}
 \bbb{f}_{pressure} = \frac{p_{0}}{V(\solu)} \unitnormal,
 \end{equation}
 where $V(\solu) = \max \{ \int_{\surf} \frac{1}{3} \solu \cdot \unitnormal d\sigma, \, 0 \}$ is an approximation of the volume of $\Omega$ and $p_{0}$ is a pressure coefficient so that $p_{0} / |V(\solu)|$ is the pressure difference between interior and exterior of the cell.
 
 \item {\em Coupling between membrane and cortex}: Forces arise due to molecules connecting the membrane with the actin cortex. When the membrane detaches during the blebbing process then these linkers break and an actin scar is left behind. In the longer run, it disintegrates and the cortex reassembles close to the new membrane position. As we are interested in the faster bleb formation we assume the cortex to be stationary and positioned a small distance $l_{0}$ away from the initial membrane. Connection points of linkers in the cortex are given by $\solu_{c} = \bbb{id}_{\surf} - l_{0} \unitnormal$ where $\bbb{id}_{\surf}$ is the identity map on $\surf$. The linker molecules can be modelled as the density of simple springs with parameter $k_{l}$ and assumed to be initially at rest, resulting in the energy density $e_{coupling} = \tfrac{k_{l}}{2} ( |\solu - \solu_{c}| - l_{0} )^2$ as long as they are intact. But as a critical length $u_{B}$ is exceeded they break, and when they get closer than a distance $u_{R}$ to the cortex then the repulsion force is increased to prevent any intersection. A model for the force then reads 
 \begin{equation} \label{eq:coupling_spec}
 \bbb{f}_{coupling} = - k_{coupling}(|\solu - \solu_{c}|) \Big{(} (\solu - \solu_{c}) - l_{0} \frac{(\solu - \solu_{c})}{|\solu - \solu_{c}|} \Big{)}, 
 \end{equation}
 with 
 \begin{equation} \label{eq:k_coupling}
 k_{coupling}(y) = k_{l} \big{(} 1 + k_{L} H(u_{R} - y) \big{)} H (u_{B} - y) 
 \end{equation}
 with some constant $k_{L}>0$ and the Heaviside function $H(r)=1$ if $r \geq 0$ and $H(r)=0$ otherwise. 
 
 \item {\em Abstract coupling including pressure}: In the abstract model, instead of \eqref{eq:coupling_spec} and \eqref{eq:force_pressure} we consider a force given by some function $\bbb{k} : \surf \times \R^3 \to \R^3$. The dependence on the first argument enables to account for given data such as the position of the cortex or the unit normal. We assume that $\bbb{k}$ is bounded and measurable with respect to the first argument and uniformly Lipschitz continuous in the second argument, i.e., there is some constant $C_{k}>0$ such that for all $\bbb{y} \in \surf$ 
 \begin{equation} \label{eq:kLipschitz}
 |\bbb{k}(\bbb{y},\bbb{a}) - \bbb{k}(\bbb{y},\bbb{b})| \leq C_{k} |\bbb{a} - \bbb{b}| \quad \forall \bbb{a},\bbb{b} \in \mbR^3. 
 \end{equation}
 This implies that $|\bbb{k}(\bbb{y},\bbb{a})| \leq C_{k} |\bbb{a}| + C$ for some constant $C>0$. The force in a point $\bbb{y} \in \surf$ is then
 \begin{equation} \label{eq:coupling_abs}
 \big{(} \bbb{f}_{coupling} + \bbb{f}_{pressure} \big{)} (\bbb{y}) = \bbb{k}(\bbb{y},\solu(\bbb{y})). 
 \end{equation}
 Note that the specific model \eqref{eq:coupling_spec}, \eqref{eq:force_pressure} does not satisfy the Lipschitz continuity condition on $\bbb{k}$. However, smoothing the Heaviside function and ensuring that the denominators do not degenerate is sufficient. With a small parameter $\eps>0$ the choice 
 \[
 k_{coupling,\eps}(z) = k_{l} \Big{(} 1 + \frac{k_{L}}{1 + \exp(2(z-u_{R})/\eps)} \Big{)} \frac{1}{1 + \exp(2(z-u_{B})/\eps)} 
 \]
 and then 
 \begin{equation} \label{eq:fcp_eps}
 \bbb{k}(\bbb{y},\solu) = - k_{coupling,\eps}(|\solu - \solu_{c}(\bbb{y})|) \Big{(} 1 - \frac{l_{0}}{|\solu - \solu_{c}(\bbb{y})| + \eps} \Big{)} (\solu - \solu_{c}(\bbb{y})) + \frac{p_{0}}{V(\solu) + \eps} \unitnormal(\bbb{y})
 \end{equation}
 is such that the assumptions are satisfied again. 
 
 \item {\em Tension}: Membranes under tension may also be modelled with linear springs, leading to an energy density of the form 
 \begin{equation} \label{eq:energy_tension}
 e_{tension} = \frac{k_{\psi}}{2} \big{(} | \snab \solu | - \sqrt{2} x_{0} \big{)}^2,
 \end{equation}
 where $k_{\psi}$ is a spring parameter and $x_{0}$ is the resting length. Note that $|\snab \solu(\cdot,0)| = | \snab \bbb{id}_{\surf}| = |\bbb{P}| = \sqrt{2}$, whence in case $x_{0}=1$ the membrane initially is at rest. The energy leads to the membrane (tension) force 
 \begin{equation} \label{eq:tension_spec}
 f_{tension} = - \partial e_{tension} = k_{\psi} \snab \cdot \Big{(} \snab \solu - \sqrt{2} x_{0} \frac{\snab \solu}{| \snab \solu |} \Big{)}.
 \end{equation}

 \item {\em Abstract tension}: In the abstract model, instead of \eqref{eq:energy_tension} we consider a tension energy with a density $\psi : \surf \times \R^{3 \times 3} \to [0,\infty)$. We assume that $\psi$ is bounded and measurable with respect to the first argument and continuously differentiable with respect to the second argument with uniformly Lipschitz continuous partial derivative, i.e., denoting with $\psi'$ this ($3 \times 3$ tensor-valued) partial derivative we assume that there is a constant $C_{\psi}>0$ such that for all $\bbb{y} \in \surf$ 
 \begin{equation} \label{eq:psipLipschitz}
 | \psi'(\bbb{y},\bbb{A}) - \psi'(\bbb{y},\bbb{B}) | \leq C_{\psi} |\bbb{A}-\bbb{B}| \quad \forall \bbb{A},\bbb{B} \in \mbR^{3 \times 3}. 
 \end{equation}
 This implies that $|\psi'(\bbb{y},\bbb{A})| \leq C_{\psi} |\bbb{A}| + C$ for some constant $C>0$. The corresponding force in a point $\bbb{y} \in \surf$ reads 
 \begin{equation} \label{eq:tension_abs}
 \bbb{f}_{tension}(\bbb{y}) = \snab \cdot \psi'(\bbb{y},\snab \solu). 
 \end{equation}
 Formally, this requires $\psi'$ to be differentiable. However, for the variational formulation that is analysed in this paper the above assumptions are sufficient. \\
 As for the coupling and the pressure term, the specific tension model \eqref{eq:energy_tension} does not satisfy the regularity requirements on $\psi$ but replacing $| \snab \solu |$ by $\sqrt{| \snab \solu |^2 + \eps}$ with some small $\eps>0$ does, and then 
 \begin{equation} \label{eq:psip_eps}
 \psi'(\bbb{y},\bbb{A}) = k_{\psi} \Big{(} 1 - \frac{\sqrt{2} x_{0}}{\sqrt{| \bbb{A} |^2 + \eps}} \Big{)} \bbb{A}. 
 \end{equation}

 \item {\em Regularisation}: The membrane resists bending, though much less than stretching. The corresponding elastic energy may be modelled as in \cite{Hel_ZNaturfC_1973}. However, its impact on the blebbing site selection and its shape has been found to be significantly smaller than that of the tension \cite{TysZatBreKay_PNAS_2014}. We therefore choose a simplified linear model that may be considered as a regularization: $e_{reg} = \tfrac{k_{b}}{2} | \Delta_{\surf} \solu |^2$ where $k_{b}$ is a (small) bending resistance coefficient, so that the regularization force is given by 
 \[
 \bbb{f}_{reg} = - \partial e_{reg} = k_{b} \Delta_{\surf}^2 \solu = - k_{b} \Delta_{\surf} \solw, \quad \mbox{where } \solw = - \Delta_{\surf} \solu
 \]
 will be referred to as curvature in the following. 
 
 \item {\em Viscous drag}: The (viscous) fluid motion in the interior and exterior of the cell is not explicitly modeled but only accounted for by a viscous drag force that opposes any membranes movement:
 \[
 \bbb{f}_{drag} = - \omega \pd_{t} \solu,
 \]
 where $\omega$ is an effective material parameter related to the viscosity of the ambient fluid. 

\end{itemize}

With the abstract choices for tension \eqref{eq:tension_abs} and coupling \eqref{eq:coupling_abs} the force balance \eqref{eq:forcebal} results in the PDE
\begin{equation} \label{eq:PDEstrong}
 \omega \pd_t \solu + k_{b} \Delta_{\surf}^2 \solu - \snab \cdot \psi'(\snab \solu) + \bbb{k}(\solu) = 0. 
\end{equation}

The model with the specific choices \eqref{eq:tension_spec} and \eqref{eq:coupling_spec} has been used for numerical simulations in Section \ref{sec:geomsim}. It has been non-dimensionalised by choosing a length scale $U$ and $k_{\psi}$ as an energy density scale. Choosing the time scale $T = U^2 \omega / k_{\psi}$ then eliminates the viscosity parameter. Writing again $\surf$, $\solu$, $\solu_{c}$, $u_{B}$, $u_{R}$, $l_{0}$, and $V(\solu)$ for the respective non-dimensional objects we obtain the equation 
\begin{align}
 \pd_t \solu =& - \lambda_{b} \Delta_{\surf}^2 \solu + \snab \cdot \Big{(} \snab \solu - \sqrt{2} x_{0} \frac{\snab \solu}{|\snab \solu|} \Big{)} \label{eq:PDEnondim} \\
 & - \lambda_{l} \big{(} 1 + k_{L} H(u_{R} - |\solu - \solu_{c}|) \big{)} H (u_{B} - |\solu - \solu_{c}|) \Big{(} (\solu - \solu_{c}) - l_{0} \frac{\solu - \solu_{c}}{|\solu - \solu_{c}|} \Big{)} + \frac{\lambda_{p}}{V(\solu)} \unitnormal. \nonumber
\end{align}
with the non-dimensional parameters $\lambda_{b} = k_{b} / (U^2 k_{\psi})$, $\lambda_{l} = k_{l} U^2 / k_{\psi}$, and $\lambda_{p} = p_{0} / (U^2 k_{\psi})$ (noting that $x_{0}$ and $k_{L}$ were non-dimensional already). The model in \cite{TysZatBreKay_PNAS_2014,ColEtAl_Nature_2017} is obtained by reducing the dimension of this equation \eqref{eq:PDEnondim} (i.e., $\surf$ is a curve in 2D). The curve then is parametrized by arc-length, and their computational model is obtained by using standard finite difference techniques.

\subsection{Variational formulation}

We aim for approximating the PDE problem using finite elements and thus require a variational formulation. Writing $\snab = (\underline{D}_1,\underline{D}_2,\underline{D}_3)$ for the components of the surface gradient, we say that a function $f \in L^1(\surf)$ has a weak derivative $\eta_i = \underline{D}_i f \in L^1(\surf)$ if 
\[
 \int_{\surf} f \underline{D}_i \varphi d\sigma = -\int_{\surf} \eta_i \varphi d\sigma + \int_{\surf} f \varphi \bbb{\kappa}_{i} d\sigma, \quad i=1,2,3, 
\] 
holds true for all smooth functions $\varphi$ with compact support. We use $\snab$ again to denote this weak derivative. Sobolev spaces on $\surf$ are defined by $H^{0}(\surf) = L^{2}(\surf)$ and 
\[
 H^{k} = H^{k}(\surf) = \big{\{} \eta \in H^{k-1}(\surf) \, \big{|} \, \snab \eta \in L^{2}(\surf) \big{\}}. 
\]
On these we will consider the Bochner spaces 
\[
 L^{2}_{H^{k}} = \big{\{} \zeta : (0,T) \to H^{k} \, \big{|} \, \int_{0}^{T} \| \zeta(t) \|_{H^{k}}^{2} dt < \infty \big{\}}, \quad 
 L^{\infty}_{H^{k}} = \big{\{} \zeta : (0,T) \to H^{k} \, \big{|} \, \esssup_{t\in(0,T)} \| \zeta(t) \|_{H^{k}} < \infty \big{\}}. 
\]
For the $L^{2}$ 'mass' inner product of vector valued functions $\bbb{v},\bbb{z} \in (L^{2})^3$ we write
\[
 m(\bbb{v},\bbb{z}) = \int_{\surf} \bbb{v} \cdot \bbb{z} d\sigma. 
\]
Note that, thanks to the Lipschitz assumption on $\bbb{k}$ also $m(\bbb{k}(\bbb{v}),\bbb{z})$ is well-defined for $\bbb{v},\bbb{z} \in (L^{2}(\surf))^3$. For the $H^{1}$ 'stiffness' semi-inner product of vector valued functions $\bbb{v},\bbb{z} \in (H^{1}(\surf))^3$ we write
\[
 s(\bbb{v},\bbb{z}) = \int_{\surf} \snab \bbb{v} \colon \snab \bbb{z} d\sigma 
\]
where $\bbb{A}\colon\bbb{B} = \sum_{i,j=1}^3 \bbb{A}_{i,j} \bbb{B}_{i,j}$ for tensors $\bbb{A},\bbb{B} \in \mbR^{3 \times 3}$. With a slight abuse of notation we also define 
\[
 s(\psi';\bbb{v},\bbb{z}) = \int_\surf \psi'(\snab \bbb{v}) : \snab \bbb{z} d\sigma. 
\]
For the (weak) variational formulation of \eqref{eq:PDEstrong} we assume without loss of generality that $\omega = 1$ and $k_b = 1$. 
\begin{problem} \label{prob:weak}
 Find $\solu,\solw \in L^2(0,T;H^1(\surf))$ with $\pd_t \solu \in L^2(0,T;L^2(\surf))$ such that for all $\tfu,\tfw \in H^1(\surf)$ and almost all $t \in (0,T)$
 \begin{align}
 m(\pd_t \solu,\tfu) + s(\solw,\tfu) + s(\psi';\solu,\tfu) + m(\bbb{k}(\solu),\tfu) &= 0, \label{eq:weak1} \\
 s(\solu,\tfw) - m(\solw,\tfw) &= 0, \label{eq:weak2} 
 \end{align}
 and such that $\solu(\cdot,0) = \bbb{id}_\surf$ is the identic map of $\surf$. 
\end{problem}

%%%%%%%%%%%%%%%%%%%%%%%%%%%%%%%%%%%%%%%%%%%%%%%%%%%%%%%%%%%%%%%
\section{Surface finite element approach}
\label{sec:sfem}

\subsection{Surface triangulations and finite elements}

The membrane $\surf$ is approximated by a family of polyhedral surfaces $\{ \surfh \}_{h}$, each one being of the form 
\begin{align*}
 \surfh = \bigcup_{E \in \fff{T}_{h}} E \subset \mathbb{R}^3
\end{align*}
where the $E$ are closed, flat non-degenerate triangles whose pairwise intersection is a complete edge, a single point, or empty. 
For each $E$ belonging to the set $\fff{T}_{h}$ of triangles we denote by $h(E) =\diam(E)$ its diameter and then identify $h = \max_{E \in \fff{T}_{h}} h(E)$ with the maximal edge length of the whole triangulation. We assume that the vertices of $\surfh$ belong to $\surf$ so that $\surfh$ is a piecewise linear interpolation of $\surf$. We also assume that $h$ is small enough so that $\surfh$ lies in the thin layer around $\surf$ in which the signed distance function $d$ is well-defined. Furthermore, we assume that $\surfh$ is the boundary of a domain that approximates $\Omega(0)$ and denote the external unit normal, which is defined on the triangles and thus piecewise constant, with $\unitnormalh$. By $\bbb{P}_{h} = \bbb{I} - \unitnormalh \otimes \unitnormalh$ we denote the projection to the tangent space in points on $\surfh$ where it exists (i.e., in the interiors of the triangles $E \in \fff{T}_{h}$). Following \eqref{eq:def_snab} this gives rise to the piecewise (i.e., triangle by triangle) definition of a surface gradient $\snabh$ on $\surfh$. The same notation is used again for the weak derivative. We write $d\sigma_{h}$ for the surface area element when integrating functions on $\surfh$. 

For the error analysis we have to measure the distance of functions such as $\solu$ on $\surf$ to functions such as the finite element solution on $\surfh$. For this purpose, consider the bijection given defined by 
\begin{equation} \label{eq:bijection}
 \surfh \ni \bbb{y}_{h} = \bbb{y} + d(\bbb{y}_{h}) \unitnormal(\bbb{y}), \quad \bbb{y} \in \surf.
\end{equation}
This bijection gives rise to the \emph{lift} of any function $\eta : \surfh \to \mbR$ to $\surf$ defined by 
\[
 \eta^\ell : \surf \to \mbR, \quad \eta^\ell(\bbb{y}) = \eta(\bbb{y}_{h}). 
\]
Writing $\mu_{h}$ for the local change of the surface area element, i.e., $d\sigma_{h} = \mu_{h} d\sigma$, integrals transform as 
\begin{equation} \label{eq:transint}
 \int_{\surfh} \eta d\sigma_{h} = \int_{\surf} \eta^\ell \mu_{h} d\sigma. 
\end{equation}
A straightforward calculation show that in points where both $\eta$ and $\eta^\ell$ are differentiable 
\begin{equation} \label{eq:transdiff}
 \snabh \eta (\bbb{y}_{h}) = \bbb{Q}_{h}(\bbb{y}) \snab \eta^\ell(\bbb{y}) \quad \mbox{where} \quad \bbb{Q}_{h}(\bbb{y}) = \bbb{P}_{h}(\bbb{y}_{h}) (\bbb{I} - d(\bbb{y}_{h}) \mmm{H}(\bbb{y})) \bbb{P}(\bbb{y}). 
\end{equation}
The following two lemmas on the errors due to the approximation of the surface and on the stability of the lift are due to \cite{Dzi_PDECV_1988,DziEll_JComputMath_2007}. 

\begin{lemma} \label{lem:geomest}
The following estimates hold true for some constant $C>0$ independent of $h$: 
\begin{align*}
\left\| 1-\mu_{h} \right\|_{L^\infty(\surf)} & \leq C h^2, \\
\left\| \bbb{Q}_{h} - \bbb{P} \right\|_{L^\infty(\surf)} & \leq C h. 
\end{align*}
\end{lemma}

\begin{lemma} \label{lem:lifting}
Let $\eta : \surfh \to \mbR$ with its lifted counterpart $\eta^\ell : \surf \to \mbR$. Let also $E \in \fff{T}_{h}$ and $E^\ell = \{ \bbb{y} \in \surf \, | \, \bbb{y}_{h} \in E \}$. The following estimates hold true with a constant $C>0$ independent of $h$ and the element $E$: 
\begin{alignat*}{6}
& \frac{1}{C} \left \| \eta^\ell \right \|_{L^2(E^\ell)} && \leq \left \| \eta \right \|_{L^2(E)} && \leq C \left \| \eta^\ell \right \|_{L^2(E^\ell)}, \\
& \frac{1}{C} \left \| \snab \eta^\ell \right \|_{L^2(E^\ell)} && \leq \left \| \snabh \eta \right \|_{L^2(E)} && \leq C \left \| \snab \eta^\ell \right \|_{L^2(E^\ell)}. 
\end{alignat*}
These inequalities generalize to the whole surfaces by summing over the elements. 
\end{lemma}

The standard finite element space used throughout is
\begin{align*}
 S_{h} = \{ \phi_{h} \in C^0 (\surfh) \, | \, \phi_{h}|_E \text{ is linear for each } E \in \fff{T}_{h} \}.
\end{align*}
Note that the identity map of $\surfh$ belongs to $S_{h}^3$. Bilinear forms corresponding to $m$ and $s$ are defined for finite element functions $\bbb{R}_{h},\bbb{Z}_{h} \in S_{h}^3$ on the triangulation by 
\begin{align*}
 m_{h}(\bbb{R}_{h},\bbb{Z}_{h}) = \int_{\surfh} \bbb{R}_{h} \cdot \bbb{Z}_{h} d\sigma_{h}, \quad 
 s_{h}(\bbb{R}_{h},\bbb{Z}_{h}) = \int_{\surfh} \snabh \bbb{R}_{h} \colon \snabh \bbb{Z}_{h} d\sigma_{h}, 
\end{align*}
and we will also use again the notation $s_{h}(\psi';\bbb{R}_{h},\bbb{Z}_{h}) = \int_{\surfh} \psi'(\snabh \bbb{R}_{h}) : \snabh \bbb{Z}_{h} d\sigma_{h}.$ For the discrepancy to the forms on $\surf$ we note the following result: 

\begin{lemma}[\cite{Dzi_PDECV_1988}] \label{lem:compforms} There is a constant $C>0$ independent of $h$ such that for all $\bbb{R}_{h}, \bbb{Z}_{h} \in S_{h}^3$ 
\begin{align*}
 | m_{h}(\bbb{R}_{h},\bbb{Z}_{h}) - m(\bbb{r}_{h},\bbb{z}_{h})| &\leq C h^2 \| \bbb{R}_{h} \|_{L^2(\surfh)} \| \bbb{Z}_{h} \|_{L^2(\surfh)}, \\
 | s_{h}(\bbb{R}_{h},\bbb{Z}_{h}) - s(\bbb{r}_{h},\bbb{z}_{h})| &\leq C h^2 \| \snabh \bbb{R}_{h} \|_{L^2(\surfh)} \| \snabh \bbb{Z}_{h} \|_{L^2(\surfh)}, 
\end{align*}
where $\bbb{r}_{h}=\bbb{R}_{h}^\ell$ and $\bbb{z}_{h}=\bbb{Z}_{h}^\ell$. 
\end{lemma}

We define the \emph{Ritz projection} $\Pi_{h} : H^1(\surf) \to S_{h}$ by 
\[
 s_{h}(\Pi_{h} (\xi), \phi_{h}) = s(\xi, \phi_{h}^\ell) \quad \forall \phi_{h} \in S_{h}, \qquad \int_{\surfh} \Pi_{h}(\xi) d\sigma_{h} = \int_{\surf} \xi d\sigma. 
\]
It's lift is denoted by $\pi_{h}(\xi) = \Pi_{h}(\xi)^\ell$ and has the following approximation properties: 

\begin{lemma}[\cite{Dzi_PDECV_1988}] \label{lem:Ritz} If $\xi \in H^1(\surf)$ then 
\[
 \| \xi - \pi_{h}(\xi) \|_{H^1(\surf)} \to 0, \quad \| \xi - \pi_{h}(\xi) \|_{L^2(\surf)} \leq C h \| \xi \|_{H^1(\surf)},
\]
and if $\xi \in H^2(\surf)$ then 
\[
 \| \xi - \pi_{h}(\xi) \|_{L^2(\surf)} + h \| \snab (\xi - \pi_{h}(\xi)) \|_{L^2(\surf)} \leq C h^2 \| \xi \|_{H^2(\surf)}
\]
where $C>0$ is a constant independent of $h$ and $\xi$. \\
The projection and the convergence results extend to functions in $L^2_{H^1}$ with a pointwise (in time) definition of the projection and with the norms $\| \cdot \|_{L^2_{H^k}}$. 
\end{lemma}

\subsection{Semi-discrete problem}

In applications, we may only have access to a triangulated surface $\surfh$ but not $\surf$, for instance, when $\surfh$ is computed from image data. Fields such as $\unitnormal$ or $\solu_{c}$ then are only approximately known in terms of $\unitnormalh$ or $\solu_{c,h}= \bbb{id}_{\surfh} - l_{0} \unitnormalh$, too. We therefore assume that the force due to coupling and pressure is given by some function (properly, a $h$ family of functions) $\bbb{k}_{h} : \surfh \times \mbR^3 \to \mbR^3$ that has the same regularity properties as $\bbb{k}$. In particular, $\bbb{k}_{h}$ is Lipschitz continuous in the second argument with the same Lipschitz constant $C_k>0$ independently of $h$. Using \eqref{eq:bijection} we define its lift $\bbb{k}_{h}^\ell : \surf \times \mbR^3 \to \mbR^3$ by $\bbb{k}_{h}^\ell (\bbb{y},\bbb{a}) = \bbb{k}_{h}(\bbb{y}_{h},\bbb{a})$, $\bbb{a} \in \mbR^3$. We assume that $\bbb{k}_{h}$ is an approximation of $\bbb{k}$ in the following sense: There is a constant $C>0$ independent of $h$ such that for all $\bbb{a} \in \mbR^3$
\begin{equation} \label{eq:ass_cons}
 \| \bbb{k}(\cdot,\bbb{a}) - \bbb{k}_{h}^\ell(\cdot,\bbb{a}) \|_{L^\infty(\surf)} \leq C (1 + |\bbb{a}|) h.
\end{equation}
With regards to the specific model \eqref{eq:fcp_eps}, the approximation
\begin{multline*}
 \bbb{k}_{h}(\bbb{y}_{h},\soluh) = - k_{coupling,\eps}(|\soluh - \solu_{c,h}(\bbb{y}_{h})|) \Big{(} (\soluh - \solu_{c,h}(\bbb{y}_{h})) - l_{0} \frac{(\soluh - \solu_{c,h}(\bbb{y}_{h}))}{|\soluh - \solu_{c,h}(\bbb{y}_{h})| + \eps} \Big{)} \\
 + \frac{p_{0}}{V_{h}(\soluh) + \eps} \unitnormalh(\bbb{y}_{h})
\end{multline*}
with 
\[
 V_{h}(\soluh) = \max \Big{\{} \int_{\surfh} \frac{1}{3} \soluh \cdot \unitnormalh d\sigma_{h}, \, 0 \Big{\}}
\]
satisfies the assumptions. 

\begin{problem} \label{prob:semidiscrete}
Find $\soluh,\solwh \in C^1(0,T;S_{h}^3) \times C^0(0,T;S_{h}^3)$ such that for all $\tfuh,\tfwh \in S_{h}^3$ and all $t \in (0,T)$
\begin{align}
 m_{h}(\pd_t \soluh,\tfuh) + s_{h}(\solwh,\tfuh) + s_{h}(\psi';\soluh,\tfuh) + m_{h}(\bbb{k}_{h}(\soluh),\tfuh) &= 0, \label{eq:semidis1} \\
 s_{h}(\soluh,\tfwh) - m_{h}(\solwh,\tfwh) &= 0, \label{eq:semidis2} 
\end{align}
and such that $\soluh(\cdot,0) = \bbb{id}_\surfh$. 
\end{problem}

In the next subsection we will show the following main result: 

\begin{theorem} \label{theo:main}
The semi-discrete problems \ref{prob:semidiscrete} are well-posed for all $h>0$ small enough. As $h \to 0$ the lifted solutions $(\soluhl,\solwhl) = (\soluh^\ell,\solwh^\ell)$ converge to some functions $(\solu,\solw)$ that uniquely solve the abstract variational problem \ref{prob:weak} and satisfy 
\begin{equation} \label{eq:stability}
 \| \solu \|_{L^\infty_{H^1}}^2 + \| \solw \|_{L^2_{H^1}}^2 \leq C 
\end{equation}
with some $C>0$ that depends on data only. 
\end{theorem}

\subsection{Proof of Theorem \ref{theo:main}} 

We generally follow the procedure in \cite{EllRan_NumerMath_2015}. Essential differences consist in the approximation of the data $\bbb{k}$ by $\bbb{k}_{h}$ and the non-linear function $\psi'$ of the gradient. To deal with the former, the consistency assumption \eqref{eq:ass_cons} will turn out sufficient, whilst for the latter we will exploit the relations \eqref{eq:weak2} and \eqref{eq:semidis2} to show strong convergence of the gradient of the deformation. 

Short time existence for \eqref{eq:semidis1}, \eqref{eq:semidis2} is straightforward to show. Estimates are now derived that are, at first, only valid at times of existence but then in the usual way can be used to show existence over the whole time interval by a continuation argument. We therefore state these estimates directly on the whole time interval. We also use the standard notion of $C>0$ for a generic constant that depends on the problem data but not on any solution, and which may change from line to line. 

Testing with $\tfuh = \soluh$ in \eqref{eq:semidis1} and $\tfwh = \solwh$ in \eqref{eq:semidis2} and subtracting these identities yields that 
\begin{align}
 \frac{1}{2} \frac{d}{dt} \| \soluh \|_{L^2}^2 + \| \solwh \|_{L^2}^2 & = - s_{h}(\psi';\soluh,\tfuh) - m_{h}(\bbb{k}_{h}(\soluh),\tfuh) \nonumber \\
 & \leq C \big{(} \| \snabh \soluh \|_{L^2}^2 + \| \soluh \|_{L^2}^2 + 1 \big{)}. \label{eq:estimhilf1}
\end{align}
Here and in the following we use the Lipschitz continuity of $\psi'$ and $\bbb{k}_{h}$, which implies linear growth (see \eqref{eq:psipLipschitz}, \eqref{eq:kLipschitz} and the comments after). Choosing $\tfwh = \soluh$ in \eqref{eq:semidis2} we see that 
\[
 \| \snabh \soluh \|_{L^2}^2 = s_{h} (\soluh, \soluh) = m_{h} (\solwh, \soluh) \leq \frac{\eps}{2} \| \solwh \|_{L^2}^2 + \frac{1}{2\eps} \| \soluh \|_{L^2}^2, 
\]
for $\eps>0$, and choosing $\eps$ small enough we thus obtain from \eqref{eq:estimhilf1} that 
\[
 \frac{1}{2} \frac{d}{dt} \| \soluh \|_{L^2}^2 + \frac{1}{2} \| \solwh \|_{L^2}^2 \leq C \big{(} \| \soluh \|_{L^2}^2 + 1 \big{)}. 
\]
A Gronwall argument therefore yields the estimate 
\begin{equation} \label{eq:estim1}
 \| \soluh \|_{L^\infty_{L^2}}^2 + \| \solwh \|_{L^2_{L^2}}^2 \leq C. 
\end{equation}

Testing with $\tfuh = \solwh$ in \eqref{eq:semidis1} and $\tfwh = \pd_{t} \soluh$ in \eqref{eq:semidis2} and then adding these equations yields that 
\[
 s_{h}(\pd_{t} \soluh, \soluh) + s_{h}(\solwh, \solwh) = - s_{h}(\psi'; \soluh, \solwh) - m_{h} (\bbb{k}_{h}(\soluh),\solwh). 
\]
With $\tfwh = \solwh$ in \eqref{eq:semidis2} we get for any small $\eps>0$ that 
\[
 \| \solwh \|_{L^2}^2 = m_{h} (\solwh,\solwh) = s_{h}(\soluh, \solwh) \leq \eps \| \snabh \solwh \|_{L^2}^2 + \frac{1}{4\eps} \| \soluh \|_{L^2}^2.
\]
Using the Lipschitz continuity of $\psi'$ and $\bbb{k}_{h}$ again we thus can conclude that 
\begin{align*}
 \frac{1}{2} \frac{d}{dt} & \| \snabh \soluh \|_{L^2}^2 + \| \snabh \solwh \|_{L^2}^2 \\
 & \leq \frac{1}{2} \| \psi'(\snabh \soluh) \|_{L^2}^2 + \frac{1}{2} \| \snabh \solwh \|_{L^2}^2 + \frac{1}{2} \| \bbb{k}_{h}(\soluh) \|_{L^2}^2 + \frac{1}{2} \| \solwh \|_{L^2}^2 \\
 & \leq \frac{1 + \eps}{2} \| \snabh \solwh \|_{L^2}^2 + C \big{(} \| \soluh \|_{L^2}^2 + \| \snabh \soluh \|_{L^2}^2 + 1 \big{)}. 
\end{align*}
Choosing $\eps$ small enough and then applying \eqref{eq:estim1} and a Gronwall argument we obtain the estimate 
\begin{equation} \label{eq:estim2}
 \| \snabh \soluh \|_{L^\infty_{L^2}}^2 + \| \snabh \solwh \|_{L^2_{L^2}}^2 \leq C. 
\end{equation}

Taking the time derivative in \eqref{eq:semidis2} (which also implies that $\pd_{t} \solwh$ exists) yields that $s_{h}(\pd_{t} \soluh, \tfwh) = m_{h}(\pd_{t} \solwh, \tfwh)$. We test this with $\tfwh = \solwh$ and subtract it from \eqref{eq:semidis1} with $\tfuh = \pd_{t} \soluh$ to obtain that 
\[
 m_{h}(\pd_{t} \soluh, \pd_{t} \soluh) + m_{h}(\pd_{t} \solwh, \solwh) + s_{h}(\psi'; \soluh, \pd_{t} \soluh) + m_{h}(\bbb{k}_{h}(\soluh), \pd_{t} \soluh) = 0. 
\]
Noting that 
\[ 
 s_{h}(\psi'; \soluh, \pd_{t} \soluh) = \int_{\surfh} \psi'(\snabh \soluh) : \pd_{t} \snabh \soluh d\sigma_{h} = \int_{\surfh} \frac{d}{dt} \psi(\snabh \soluh) d\sigma_{h}
\]
and using the Lipschitz continuity of $\bbb{k}_{h}$ again we see that 
\[
 \| \pd_{t} \soluh \|_{L^2}^2 + \frac{1}{2} \frac{d}{dt} \| \solwh \|_{L^2}^2 + \frac{d}{dt} \Big{(} \int_{\surfh} \psi(\snabh \soluh) d\sigma_{h} \Big{)} \leq C (\| \soluh \|_{L^2}^2 + 1) + \frac{1}{2} \| \pd_{t} \soluh \|_{L^2}^2. 
\]
Therefore, with \eqref{eq:estim1} we obtain the estimate 
\begin{equation} \label{eq:estim3}
 \| \pd_{t} \soluh \|_{L^2_{L^2}}^2 + \| \solwh \|_{L^\infty_{L^2}}^2 + \sup_{t \in [0,T]} \int_{\surfh} \psi(\snabh \soluh) d\sigma_{h} \leq C. 
\end{equation}

These estimates \eqref{eq:estim1}--\eqref{eq:estim3} are now lifted from $\surfh$ to $\surf$. We can then apply compactness arguments to deduce the existence of limits $(\solu,\solw)$, which we will show to satisfy Problem \ref{prob:weak}. As a first step, the stability estimate \eqref{eq:stability} will be derived. Using Lemma \ref{lem:lifting} the lifted solutions satisfy the estimates 
\begin{align}
 \| \soluhl \|_{L^\infty_{H^1}}^2 + \| \solwhl \|_{L^2_{H^1}}^2 &\leq C, \label{eq:estiml2} \\
 \| \pd_{t} \soluhl \|_{L^2_{L^2}}^2 + \| \solwhl \|_{L^\infty_{L^2}}^2 &\leq C. \label{eq:estiml3}
\end{align}
Hence, there are functions $\solu \in L^2_{H^1}$ with $\pd_{t} \in L^2_{L^2}$ and $\solw \in L^2_{H^1}$ such that for a subsequence as $h \to 0$ 
\begin{align}
 \soluhl &\rightharpoonup \solu & \quad & \mbox{in } L^2_{H^1}, \qquad & \pd_t \soluhl &\rightharpoonup \pd_t \solu & \quad & \mbox{in } L^2_{L^2}, \label{eq:convsol1} \\
 \soluhl &\to \solu & \quad & \mbox{in } L^2_{L^2} \mbox{ and a.e.,} \qquad & \solwhl &\rightharpoonup \solw & \quad & \mbox{in } L^2_{H^1}, \label{eq:convsol2} 
\end{align}
and these limits also satisfy \eqref{eq:estiml2} and \eqref{eq:estiml3} and, thus, the stability estimate \eqref{eq:stability}. 

Let us now show that $(\solu,\solw)$ satisfies \eqref{eq:weak2}. For any $\bbb{\eta} \in L^2_{H^1}$ let $\bbb{H}_{h} = \Pi_{h}(\bbb{\eta})$ denote its Ritz projection with the lift $\bbb{\eta}_{h} = \pi_{h}(\bbb{\eta})$. Then $s_{h}(\soluh,\bbb{H}_{h}) = m_{h}(\solwh,\bbb{H}_{h})$, whence 
\begin{multline*}
 \int_0^T s(\solu,\bbb{\eta}_{h}) - m(\solw,\bbb{\eta}_{h}) dt 
= \int_0^T \big{(} s(\solu,\bbb{\eta}_{h}) - s(\soluhl,\bbb{\eta}_{h}) \big{)} dt 
+ \int_0^T \big{(} s(\soluhl,\bbb{\eta}_{h}) - s_{h}(\soluh,\bbb{H}_{h}) \big{)} dt \\
+ \int_0^T \big{(} m_{h}(\solwh,\bbb{H}_{h}) - m(\solwhl,\bbb{\eta}_{h}) \big{)} dt + \int_0^T \big{(} m(\solwhl,\bbb{\eta}_{h}) - m(\solw,\bbb{\eta}_{h}) \big{)} dt =: J_{1}+J_{2}+J_{3}+J_{4}.
\end{multline*}
By the properties of the Ritz projection (Lemma \ref{lem:Ritz}) we have that $\bbb{\eta}_{h} = \pi_{h}(\bbb{\eta}) \to \bbb{\eta}$ in $L^2_{H^1}$. Thanks to \eqref{eq:convsol1} we thus have that $J_{1} \to 0$ as $h \to 0$, and similarly $J_{4} \to 0$ thanks to \eqref{eq:convsol2}. Lemma \ref{lem:compforms} together with the estimates \eqref{eq:estim2} and \eqref{eq:estim3} ensures that $J_{2} \to 0$ and $J_{3} \to 0$ as $h \to 0$. Therefore, $(\solu,\solw)$ satisfies the following identity, which implies \eqref{eq:weak2}: 
\begin{equation} \label{eq:estimhilf2}
 \int_0^T \big{(} s(\solu,\bbb{\eta}) - m(\solw,\bbb{\eta}) \big{)} dt = 0 \quad \forall \bbb{\eta} \in L^2_{H^1}.
\end{equation}

Next, we show strong convergence of $\snab \soluhl$. We note that 
\[
 \| \snab (\solu - \soluhl) \|_{L^2_{L^2}}^2 = \int_0^T s(\solu - \soluhl,\solu - \pi_{h}(\solu)) dt + \int_0^T s(\solu -\soluhl, \pi_{h}(\solu) - \soluhl) dt =: K_{1}+K_{2}. 
\]
Using again Lemma \ref{lem:Ritz} we see that $\pi_{h}(\solu) \to \solu$ in $L^2_{H^1}$, and with \eqref{eq:convsol1} this implies that $K_{1} \to 0$. Regarding the second term we note that thanks to \eqref{eq:estimhilf2} and \eqref{eq:semidis2} 
\begin{align*}
 K_{2} &= \int_0^T \big{(} m(\solw,\pi_{h}(\solu) - \soluhl) - m(\solwhl,\pi_{h}(\solu) - \soluhl) \big{)} dt \\
 & \quad + \int_0^T \big{(} m(\solwhl,\pi_{h}(\solu) - \soluhl) - m_{h}(\solwh,\Pi_{h}(\solu) - \soluh) \big{)} dt \\
 & \quad \quad + \int_0^T \big{(} s_{h}(\soluh,\Pi_{h}(\solu) - \soluh) - s(\soluhl,\pi_{h}(\solu) - \soluhl) \big{)} dt = K_{21}+K_{22}+K_{23}. 
\end{align*}
As both $\pi_{h}(\solu) \to \solu$ and $\soluhl \to \solu$ by \eqref{eq:convsol2} we see that $\pi_{h}(\solu) - \soluhl \to 0$ in $L^2_{L^2}$ as $h \to 0$. With $\solwhl \rightharpoonup \solw$ in the same space we obtain that $K_{21} \to 0$. From the definition and properties of the Ritz projection it easily follows that $\| \Pi_{h}(\xi) \|_{H^1} \leq C \| \xi \|_{H^1}$ with some $C>0$ independent of $h$ and $\xi \in H^1(\surf)$. The stability estimate \eqref{eq:stability}, which is already proved, and the estimates \eqref{eq:estim1} and \eqref{eq:estim2} therefore yield that $\| \Pi_{h}(\solu) - \soluh \|_{H^1}$ is uniformly bounded in $h$. Using \eqref{eq:estim1} and \eqref{eq:estim2} again for $\solwh$ and Lemma \ref{lem:compforms} we obtain that $K_{22} \to 0$ and $K_{23} \to 0$ as $h \to 0$. This finally shows that 
\begin{equation} \label{eq:convsol3}
 \soluhl \to \solu \quad \mbox{in } L^2_{H^1} \mbox{ and a.e.}
\end{equation}

To conclude the proof of Theorem \ref{theo:main} we need to show that $(\solu,\solw)$ satisfies \eqref{eq:weak1}. For any $\tfu \in L^2_{H^1}$ let $\tfuh = \Pi_{h}(\tfu)$ be its Ritz projection with lift $\tfuhl = \pi_{h}(\tfu)$. Then 
\begin{align}
 \int_0^T & \big{(} m(\pd_{t} \solu,\tfu) - m_{h}(\pd_{t} \soluh,\tfuh) \big{)} dt \nonumber \\
 = & \, \int_0^T \big{(} m(\pd_{t} \solu,\tfu) - m(\pd_{t} \soluhl,\tfu) \big{)} dt + \int_0^T \big{(} m(\pd_{t} \soluhl,\tfu) - m(\pd_{t} \soluhl,\tfuhl) \big{)} dt \nonumber \\
 & + \int_0^T \big{(} m(\pd_{t} \soluhl,\tfuhl) - m_{h}(\pd_{t} \soluh,\tfuh) \big{)} dt \quad =: L_{1}+L_{2}+L_{3}. \label{eq:splitterm1}
\end{align}
Thanks to \eqref{eq:convsol2} we have that $L_{1} \to 0$ as $h \to 0$. Lemma \ref{lem:Ritz} on the Ritz projection ensures that $L_{2} \to 0$. It also ensures that $\tfuh$ is uniformly bounded in $h$, and with Lemma \ref{lem:compforms} and \eqref{eq:estim3} we obtain that $L_{3} \to 0$. Altogether 
\begin{equation} \label{eq:convterm1}
 \int_0^T m_{h}(\pd_{t} \soluh,\tfuh) dt \to \int_0^T m(\pd_{t} \solu,\tfu) dt. 
\end{equation}
Analogously one can show that 
\begin{equation} \label{eq:convterm2}
 \int_0^T s_{h}(\solwh,\tfuh) dt \to \int_0^T s(\solw,\tfu) dt. 
\end{equation}

Next, we can write 
\begin{align}
 \int_0^T \big{(} & s(\psi';\solu,\tfu) - s_{h}(\psi';\soluh,\tfuh) \big{)} dt \notag \\
 & = \int_0^T \int_{\surf} \big{(} \psi'(\snab \solu) : \snab \tfu - \psi'(\snab \soluhl) : \snab \tfu \big{)} d\sigma dt \notag \\
 & \quad + \int_0^T \int_{\surf} \big{(} \psi'(\snab \soluhl) : \snab \tfu - \psi'(\snab \soluhl) : \snab \tfuhl \big{)} d\sigma dt \notag \\
 & \quad + \int_0^T \Big{(} \int_{\surf} \psi'(\snab \soluhl) : \snab \tfuhl d\sigma - \int_{\surfh} \psi'(\snabh \soluh) : \snabh \tfuh d\sigma_{h} \Big{)} dt \notag \\
 & =: M_{1}+M_{2}+M_{3}. \label{eq:splitterm3}
\end{align}
Thanks to \eqref{eq:convsol3} and the Lipschitz continuity of $\psi'$ we have that $\psi'(\snab \soluhl) \to \psi'(\snab \solu)$ in $L^2_{L^2}$ and almost everywhere, whence $M_{1} \to 0$ as $h \to 0$. For the second term we observe that 
\begin{align*}
 M_{2} &\leq \int_0^T \| \psi'(\snab \soluhl) \|_{L^2(\surf)} \| \snab \tfu - \snab \tfuhl \|_{L^2(\surf)} dt \\
 &\leq \int_0^T C \big{(} \| \snab \soluhl \|_{L^2(\surf)} + 1 \big{)} \| \tfu - \tfuhl \|_{H^1(\surf)} dt \quad \to 0
\end{align*}
thanks to the estimate \eqref{eq:estim2} and Lemma \ref{lem:Ritz}. In the last term we lift the second integral to $\surf$ (recall \eqref{eq:transint} and \eqref{eq:transdiff} for the transformation of the derivative): 
\begin{align}
 M_{3} &= \int_0^T \Big{(} \int_{\surf} \psi'(\snab \soluhl) : \snab \tfuhl d\sigma - \int_{\surf} \psi'(\bbb{Q}_{h} \snab \soluhl) : \bbb{Q}_{h} \snab \tfuhl \mu_{h} d\sigma \Big{)} dt \notag \\
 &= \int_0^T \int_{\surf} \big{(} \psi'(\snab \soluhl) - \psi'(\bbb{Q}_{h} \snab \soluhl) \big{)} : \snab \tfuhl d\sigma dt \notag \\
 & \quad + \int_0^T \int_{\surf} \psi'(\bbb{Q}_{h} \snab \soluhl) : \big{(} \bbb{P} - \mu_{h} \bbb{Q}_{h} \big{)} \snab \tfuhl d\sigma dt. \label{eq:estimM3a}
\end{align}
We can now apply the Lipschitz continuity of $\psi'$ and the geometric error estimates in Lemma \ref{lem:geomest} (which imply that $\| \bbb{Q}_{h} \|_{L^\infty(\surf)}$ is uniformly bounded in $h$) to obtain that 
\begin{align}
 |M_{3}| &\leq \int_0^T C_{\psi} \big{|} \snab \soluhl - \bbb{Q}_{h} \snab \soluhl \big{|} \, | \snab \tfuhl | dt \notag \\
 & \quad + \int_0^T C \big{(} | \bbb{Q}_{h} \snab \soluhl | + 1 \big{)} \big{(} | \bbb{P} - \bbb{Q}_{h} | + | \bbb{Q}_{h} | \big{|} 1 - \mu_{h} \big{|} \big{)} | \snab \tfuhl | dt \notag \\
 &\leq C_{\psi} \| \bbb{P} - \bbb{Q}_{h} \|_{L^\infty(\surf)} \| \snab \soluhl \|_{L^2_{L^2}} \| \snab \tfuhl \|_{L^2_{L^2}} \notag \\
 & \quad + C \big{(} \| \bbb{Q}_{h} \|_{L^\infty(\surf)} + 1 \big{)} \big{(} \| \bbb{P} - \bbb{Q}_{h} \|_{L^\infty(\surf)} + \| 1 - \mu_{h} \|_{L^\infty(\surf)} \big{)} \| \snab \soluhl \|_{L^2_{L^2}} \| \snab \tfuhl \|_{L^2_{L^2}} \notag \\
 & \leq C h \| \snab \soluhl \|_{L^2_{L^2}} \| \snab \tfuhl \|_{L^2_{L^2}}. \label{eq:estimM3b}
\end{align}
Using estimate \eqref{eq:estiml2} and that also $\| \tfuhl \|_{L^2_{H^1}} \leq C \| \tfu \|_{L^2_{H^1}}$ is uniformly bounded (follows from Lemma \ref{lem:Ritz}) we see that $M_{3} \to 0$, and we can conclude that 
\begin{equation} \label{eq:convterm3}
 \int_0^T s_{h}(\psi';\soluh,\tfuh) dt \to \int_0^T s(\psi';\solu,\tfu) dt. 
\end{equation}

For the last term in \eqref{eq:weak1} we note that 
\begin{multline}
 \int_0^T \big{(} m(\bbb{k}(\solu), \tfu) - m_{h}(\bbb{k}_{h}(\soluh), \tfuh) \big{)} dt 
 = \int_0^T \big{(} m(\bbb{k}(\solu), \tfu) - m(\bbb{k}(\soluhl), \tfu) \big{)} dt \\
 + \int_0^T m(\bbb{k}(\soluhl), \tfu - \tfuhl) dt + \int_0^T \big{(} m(\bbb{k}(\soluhl), \tfuhl) - m_{h}(\bbb{k}_{h}(\soluh), \tfuh) \big{)} dt =: N_{1}+N_{2}+N_{3}. \label{eq:splitterm4}
\end{multline}
Thanks to \eqref{eq:convsol2} and the Lipschitz continuity of $\bbb{k}$ we have that $\bbb{k}(\soluhl) \to \bbb{k}(\solu)$ in $L^2_{L^2}$ and almost everywhere, so that $N_{1} \to 0$ as $h \to 0$. The second term converges to zero thanks to $\tfuhl \to \tfu$ in $L^2_{L^2}$. Regarding $N_{3}$, we lift the second integral to $\surf$: 
\begin{align}
 N_{3} &= \int_0^T \Big{(} \int_\surf \bbb{k}(\soluhl) \cdot \tfuhl d\sigma - \int_{\surf} \bbb{k}_{h}^\ell(\soluhl) \cdot \tfuhl \mu_{h} d\sigma \Big{)} dt \notag \\
 &= \int_0^T \int_{\surf} \big{(} \bbb{k}(\soluhl) - \bbb{k}_{h}^\ell(\soluhl) \big{)} \cdot \tfuhl \notag d\sigma dt \\
 & \quad + \int_0^T \int_{\surf} (1 - \mu_{h}) \bbb{k}_{h}^\ell(\soluhl) \cdot \tfuhl d\sigma dt. \label{eq:estimN3a}
\end{align}
Using now the consistency \eqref{eq:ass_cons} of the approximation of $\bbb{k}$ by $\bbb{k}_{h}$, the Lipschitz continuity of $\bbb{k}_{h}$, and the geometric error estimates in Lemma \eqref{lem:geomest} we obtain that 
\begin{align}
 |N_{3}| &\leq \int_0^T \Big{(} \int_{\surf} C h (1 + |\soluhl|) |\tfuhl| d\sigma \Big{)} dt + \int_0^T \Big{(} \| 1 - \mu_{h} \|_{L^\infty(\surf)} \int_{\surf} C (|\soluhl| + 1) |\tfuhl| d\sigma \Big{)} dt \notag \\
 & \leq C h \big{(} 1 + \| \soluhl \|_{L^2_{L^2}} \big{)} \| \tfuhl \|_{L^2_{L^2}} \quad \to 0 \label{eq:estimN3b}
\end{align}
using estimate \eqref{eq:estiml2} and that also $\| \tfuhl \|_{L^2_{L^2}} \leq C \| \tfu \|_{L^2_{L^2}}$ is uniformly bounded. Altogether  
\begin{equation} \label{eq:convterm4}
 \int_0^T m_{h}(\bbb{k}_{h}(\soluh),\tfuh) dt \to \int_0^T m(\bbb{k}(\solu),\tfu) dt. 
\end{equation}

The convergence results \eqref{eq:convterm1}, \eqref{eq:convterm2}, \eqref{eq:convterm3}, and \eqref{eq:convterm4} show that $(\solu,\solw)$ satisfies \eqref{eq:weak1}, which is the limit of \eqref{eq:semidis1} as $h \to 0$. 

In the next section we show error estimates. The same techniques can be used to show uniqueness of the solution $(\solu,\solw)$ to Problem \ref{prob:weak}, whence we omit the details. This concludes the proof of Theorem \ref{theo:main}.

\subsection{Error estimates}

Deriving error estimates is possible when assuming higher regularity of the solution, henceforth: 
\begin{equation} \label{eq:ass_reg}
 \mbox{Assume that } \solu, \pd_{t} \solu, \solw \in L^2_{H^2}.
\end{equation}
We will derive error estimates on the triangulated surfaces and for this purpose us the bijection \eqref{eq:bijection} to \emph{anti-lift} the solution $(\solu,\solw)$ to $\surfh$: 
\[
 \solu^{-\ell}, \solw^{-\ell} : \surfh \to \mbR^3, \quad \solu^{-\ell}(\bbb{y}_{h}) = \solu(\bbb{y}), \, \solw^{-\ell}(\bbb{y}_{h}) = \solw(\bbb{y}). 
\]
We use the Ritz projection to split the errors into a projection error and a discrete error:
\begin{alignat*}{3}
 \solu^{-\ell} - \soluh &= \big{(} \solu^{-\ell} - \Pi_{h}(\solu) \big{)} + \big{(} \Pi_{h}(\solu) - \soluh \big{)} & &= \proju + \erru, \\
 \solw^{-\ell} - \solwh &= \big{(} \solw^{-\ell} - \Pi_{h}(\solw) \big{)} + \big{(} \Pi_{h}(\solw) - \solwh \big{)} & &= \projw + \errw.
\end{alignat*}
Thanks to the regularity assumption \eqref{eq:ass_reg}, the properties of the Ritz projection (Lemma \ref{lem:Ritz}), and the properties of the lift (Lemma \ref{lem:lifting}) error bounds for the projection errors are straightforward: 
\begin{alignat}{3}
 \| \proju \|_{L^2_{L^2(\surfh)}} & + h \| \snabh \proju \|_{L^2_{L^2(\surfh)}} & & \leq C h^2 \| \solu \|_{L^2_{H^2}}, \label{eq:proju_est} \\
 \| \projw \|_{L^2_{L^2(\surfh)}} & + h \| \snabh \projw \|_{L^2_{L^2(\surfh)}} & & \leq C h^2 \| \solw \|_{L^2_{H^2}}. \label{eq:projw_est}
\end{alignat}

To estimate the discrete errors we start by testing \eqref{eq:weak1} with $\tfuhl$, which is the lift of some $\tfuh \in S_{h}^3$, and then subtract \eqref{eq:semidis2} tested with $\tfuh$. Using that $s_{h}(\Pi_{h}(\solu),\tfuh) = s(\solu,\tfuhl)$ by the definition of the Ritz projection this yields that 
\begin{align}
m_{h}(\pd_{t} & \erru, \tfuh) + s_{h}(\errw,\tfuh) \notag \\
& + s_{h}(\psi';\Pi_{h}(\solu),\tfuh) - s_{h}(\psi';\soluh,\tfuh) + m_{h}(\bbb{k}_{h}(\Pi_{h}(\solu)),\tfuh) - m_{h}(\bbb{k}_{h}(\soluh),\tfuh) \notag \\
= & \, \big{(} m_{h}(\pd_{t} \Pi_{h}(\solu),\tfuh) - m(\pd_{t} \solu,\tfuhl) \big{)} \notag \\
& + \big{(} s_{h}(\psi';\Pi_{h}(\solu),\tfuh) - s(\psi';\solu,\tfuhl) \big{)} + \big{(} m_{h}(\bbb{k}_{h}(\Pi_{h}(\solu)),\tfuh) - m(\bbb{k}(\solu),\tfuhl) \big{)} \notag
\\ =: & \, E_t(\tfuh) + E_\psi(\tfuh) + E_k(\tfuh). \label{eq:ErrorEqnA}
\end{align}
Proceeding similarly with \eqref{eq:weak2} and \eqref{eq:semidis2} we obtain that 
\begin{equation}
s_{h}(\erru,\tfwh) - m_{h}(\errw,\tfwh) = m_{h}(\Pi_{h}(\solw),\tfwh) - m(\solw,\tfwhl) =: E_w(\tfwh). \label{eq:ErrorEqnB}
\end{equation}
The error terms satisfy the following estimates: 

\begin{lemma} \label{lem:EBounds}
There is some $C>0$ independent of $h$ (sufficiently small) such that for all $\tfuh,\tfwh \in S_{h}^3$
\begin{align}
|E_t(\tfuh)| \leq & \, C h^2 \| \pd_{t} \solu \|_{H^2(\surf)} \| \tfuh \|_{L^2(\surfh)}, \label{eq:EtRate} \\
|E_\psi(\tfuh)| \leq & \, C h \|\solu \|_{H^2(\surf)} \| \snabh \tfuh \|_{L^2(\surfh)}, \label{eq:EpsiRate} \\
|E_k(\tfuh)| \leq & \, C h \big{(} 1 + \| \solu \|_{H^2(\surf)} \big{)} \| \tfuh \|_{L^2(\surfh)}, \label{eq:EkRate} \\
|E_w(\tfwh)| \leq & \, C h^2 \| \solw \|_{H^2(\surf)} \| \tfwh \|_{L^2(\surfh)}. \label{eq:EwRate}
\end{align}
\end{lemma}

\begin{proof}
To show the estimates, we will frequently apply Lemma \ref{lem:compforms} on the approximation of the bilinear forms, Lemma \ref{lem:lifting} on the stability of the lift, and Lemma \ref{lem:Ritz} on the Ritz projection without explicitly pointing it out for conciseness. 

The Ritz projection commutes with the time derivative thanks to the regularity of $\solu$, whence 
\[
 E_{t}(\tfuh) = \big{(} m_{h}(\pd_{t} \Pi_{h} (\solu),\tfuh) - m(\pd_{t} \pi_{h}(\solu),\tfuhl) \big{)} + \big{(} m(\pi_{h}(\pd_{t} \solu),\tfuhl) - m(\pd_{t} \solu,\tfuhl) \big{)} =: \ttt{L}_{3} + \ttt{L}_{1}. 
\]
The term $\ttt{L}_{3}$ is similar to $L_{3}$ in \eqref{eq:splitterm1} but without the time integral and with $\pi_{h}(\pd_{t} \solu)$ instead of $\soluhl$ and thus can also be estimated similarly: 
\[
 |\ttt{L}_{3}| \leq C h^2 \| \Pi_{h}(\pd_{t} \solu) \|_{L^2(\surfh)} \| \tfuh \|_{L^2(\surfh)} \leq C h^2 \| \pd_{t} \solu \|_{H^2(\surf)} \| \tfuh \|_{L^2(\surfh)}. 
\]
Furthermore, 
\[
 |\ttt{L}_{1}| \leq \| \pi_{h}(\pd_{t} \solu) - \pd_{t} \solu \|_{L^2(\surf)} \| \tfuhl \|_{L^2(\surf)} \leq C h^2 \| \pd_{t} \solu \|_{H^2(\surf)} \| \tfuh \|_{L^2(\surfh)},
\]
which altogether yields \eqref{eq:EtRate}. 

We can also split up $E_{\psi}$: 
\[
 E_{\psi}(\tfuh) = \big{(} s_{h}(\psi';\Pi_{h}(\solu),\tfuh) - s(\psi';\pi_{h}(\solu),\tfuhl) \big{)} + \big{(} s(\psi';\pi_{h}(\solu),\tfuhl) - s(\psi';\solu,\tfuhl) \big{)} =: \ttt{M}_{3} + \ttt{M}_{1}.
\]
The first term $\ttt{M}_{3}$ is similar to the term $M_{3}$ in \eqref{eq:splitterm3}, without the time integral and with $\pi_{h}(\solu)$ instead of $\soluhl$. Following the lines of \eqref{eq:estimM3a} and \eqref{eq:estimM3b} we obtain that 
\[
 |\ttt{M}_{3}| \leq C h \| \snab \pi_{h}(\solu) \|_{L^2(\surf)} \| \snab \tfuhl \|_{L^2(\surf)} \leq C h \| \solu \|_{H^2(\surf)} \| \snabh \tfuh \|_{L^2(\surfh)}. 
\]
Using that $\psi'$ is Lipschitz, the other term is estimated as 
\[
 |\ttt{M}_{1}| \leq C_{\psi} \| \snab \pi_{h}(\solu) - \snab \solu \|_{L^2(\surf)} \| \snab \tfuhl \|_{L^2(\surf)} \leq C h \| \solu \|_{H^2(\surf)} \| \snabh \tfuh \|_{L^2(\surfh)}, 
\]
which together shows \eqref{eq:EpsiRate}. 

For the third estimate we use the splitting 
\[
 E_k(\tfuh) = \big{(} m_{h}(\bbb{k}_{h}(\Pi_{h}(\solu)),\tfuh) - m(\bbb{k}(\pi_{h}(\solu)),\tfuhl) \big{)} + \big{(} m(\bbb{k}(\pi_{h}(\solu)),\tfuhl) - m(\bbb{k}(\solu),\tfuhl) \big{)} =: \ttt{N}_{3} + \ttt{N}_{1}.
\]
Noting and exploiting the similarity of $\ttt{N}_{3}$ with $N_{3}$ in \eqref{eq:splitterm3} we proceed as in \eqref{eq:estimN3a} and \eqref{eq:estimN3b} to obtain that 
\[
 |\ttt{N}_{3}| \leq C h \big{(} 1 + \| \pi_{h}(\solu) \|_{L^2(\surf)} \big{)} \| \tfuhl \|_{L^2(\surf)} \leq C h \big{(} 1 + \| \solu \|_{H^2(\surf)}) \big{)} \| \tfuh \|_{L^2(\surfh)}. 
\]
Furthermore, 
\[
 |\ttt{N}_{1}| \leq C_{k} \| \pi_{h}(\solu) - \solu \|_{L^2(\surf)} \| \tfuhl \|_{L^2(\surf)} \leq C h^2 \| \solu \|_{H^2(\surf)} \| \tfuh \|_{L^2(\surfh)}
\]
which finally yields the estimate \eqref{eq:EkRate}. 

The last estimate \eqref{eq:EwRate} can be proved analogously to \eqref{eq:EtRate}, which concludes the proof of Lemma \ref{lem:EBounds}. 
\end{proof}

With these estimates we can derive the following estimates for the error: 

\begin{corollary} \label{cor:errestim}
Assume that $(\solu,\solw)$ solves Problem \ref{prob:weak} and satisfies $\solu,\pd_{t} \solu, \solw \in L_{H^2(\surf)}^2$. For all sufficiently small $h$ the solution $(\soluh,\solwh)$ of Problem \ref{prob:semidiscrete} satisfies 
\begin{align*}
\| \solu^{-l} - \soluh \| _{L^{\infty}_{L^2(\surfh)}}^2 + \| \solw^{-l} - \solwh \|_{L^2_{L^2(\surfh)}}^2 + \| \snabh(\solu^{-l} - \soluh) \| _{L^2_{L^2(\surfh)}}^2 \leq C h^2
\end{align*}
with a constant $C>0$ independent of $h$. 
\end{corollary}

\begin{proof}
We proceed as for deriving \eqref{eq:estim1} but start with \eqref{eq:ErrorEqnA}, where we test with $\tfuh = \erru$, and with \eqref{eq:ErrorEqnB}, where we choose $\tfwh = \errw$. Taking the difference we obtain that 
\begin{align*}
m_{h}(\pd_{t} & \erru,\erru) + m_{h}(\errw,\errw) \\
= & \, - s_{h}(\psi';\Pi_{h}(\solu),\erru) + s_{h}(\psi';\soluh,\erru) - m_{h}(\bbb{k}_{h}(\Pi_{h}(\solu)),\erru) + m_{h}(\bbb{k}_{h}(\soluh),\erru) \\
& + E_t(\erru) + E_\psi(\erru) + E_k(\erru) - E_w(\errw). 
\end{align*}
In Lemma \ref{lem:EBounds} we absorb the norm of $\solu$ into $C$ to obtain that 
\[
 |E_t(\erru)| \leq C h^4 + \frac{1}{2} \| \erru \|_{L^2(\surfh)}^2, \quad |E_\psi(\erru)| \leq C h^2 + \frac{1}{2} \| \snabh \erru \|_{L^2(\surfh)}^2, 
\]
and similarly for the other two errors. Using that $\psi'$ and $\bbb{k}_{h}$ are Lipschitz we then get that 
\begin{align}
\frac{1}{2} \frac{d}{dt} & \| \erru \|_{L^2(\surfh)}^2 + \| \errw \|_{L^2(\surfh)}^2 \notag \\
\leq & \, \int_{\surfh} C_{\psi} | \snabh \soluh - \snabh \Pi_{h}(\solu) | \, | \snabh \erru | d\sigma_{h} + \int_{\surfh} C_{k} | \soluh - \Pi_{h}(\solu) | \, | \erru | d\sigma_{h} \notag \\
& + |E_t(\erru)| + |E_\psi(\erru)| + |E_k(\erru)| + |E_w(\errw)| \notag \\
\leq & \, C_{\psi} \| \snabh \erru \|_{L^2(\surfh)}^2 + C_k \| \erru \|_{L^2(\surf)}^2 \notag \\
& + C (h^2 + h^4) + \| \erru \|_{L^2(\surfh)}^2 + \frac{1}{2} \| \snabh \erru \|_{L^2(\surfh)}^2 + \frac{1}{2} \| \errw \|_{L^2(\surfh)}^2. \label{eq:estimerr}
\end{align}
Substituting $\tfwh = \erru$ in (\ref{eq:ErrorEqnB}) gives for any $\eps>0$ that 
\begin{equation} \label{eq:estimgrad}
 \| \snabh \erru \|_{L^2(\surfh)}^2 \leq \frac{\eps}{2} \| \errw \|_{L^2(\surfh)}^2 + \frac{1}{2\eps} \| \erru \|_{L^2(\surfh)}^2 + C h^4 + \| \erru \|_{L^2(\surfh)}^2.
\end{equation}
We can thus estimate the terms involving $\| \snabh \erru \|_{L^2(\surfh)}^2$ on the right-hand-side of \eqref{eq:estimerr} by terms involving $\eps \| \errw \|_{L^2(\surfh)}^2$. Choosing now $\eps>0$ small enough, these terms involving $\| \errw \|_{L^2(\surfh)}^2$ can then be absorbed in the left-hand-side to that altogether
\[ 
 \frac{d}{dt} \| \erru \|_{L^2(\surfh)}^2 + \| \errw \|_{L^2(\surfh)}^2 \leq C \| \erru \|_{L^2(\surfh)}^2 + C h^2. 
\]
By standard interpolation theory (recall that the identic map of the triangulated surface $\surfh$ linearly interpolates the identic map of $\surf$) the initial error satisfies 
\[
 \| \erru(0) \|_{L^2(\surfh)}^2 \leq \| \proju(0) \|_{L^2(\surfh)}^2 + \| \solu^{-\ell}(0) - \soluh(0) \|_{L^2(\surfh)}^2 \leq C h^2 + \| \bbb{id}_\surf^{-\ell} - \bbb{id}_{\surfh} \|_{L^2(\surfh)}^2 \leq C h^2. 
\]
Applying Gronwall therefore yields that
\[
 \| \erru \|_{L^\infty_{L^2(\surfh)}}^2 + \| \errw \|_{L^2_{L^2(\surfh)}}^2 \leq C h^2. 
\]
From \eqref{eq:estimgrad} we now see that also 
\[
 \| \snabh \erru \|_{L^2_{L^2(\surfh)}}^2 \leq C h^2.
\]
Together with \eqref{eq:proju_est} and \eqref{eq:projw_est} these two estimates conclude the proof of Corollary \ref{cor:errestim}.  
\end{proof}

%%%%%%%%%%%%%%%%%%%%%%%%%%%%%%%%%%%%%%%%%%%%%%%%%%%%%%%%%%%%%%%
\section{Software and simulations}
\label{sec:software}

\subsection{Time discretisation}
\label{sec:prob_discrete}

In order to illustrate the capability of the computational framework that is presented and analysed in the previous section we performed some numerical simulations for the specific model \eqref{eq:PDEnondim}. Its variational form with operator splitting in Problem \ref{prob:weak} is discretised in time with a simple semi-implicit first order scheme as follows: We split the time interval $[0,T]$ into $M \in \N$ equal parts of size $\tau = T/M$, denote the time steps with $t^{(m)} = m \tau$, and write $f^{(m)} = f(t^{(m)})$ for any time dependent fields or functions. 
\begin{problem} \label{prob:discrete}
Given $\surfh$, $S_{h}^3 \ni \solu_{c,h} \approx \solu_{c}$, and parameters $\lambda_{b}$, $\lambda_{l},$ $\lambda_{p},$ $l_{0}$, $u_B$, $k_L$, $u_{R}$, for $m=0, \dots, M-1$ find $(\soluh^{(m+1)},\solwh^{(m+1)}) \in S_{h}^3 \times S_{h}^3$ such that for all $(\tfuh,\tfwh) \in S_{h}^3 \times S_{h}^3$
\begin{align}
\int_{\surfh} \frac{1}{\tau} \solu^{(m+1)} & \cdot \tfuh + \lambda_{b} \snabh \solwh^{(m+1)} : \snabh \tfuh + \snabh \soluh^{(m+1)} : \snabh \tfuh + \lambda_{coupling}^{(m)} \soluh^{(m+1)} \cdot \tfuh d\sigma_{h} \notag \\ 
= & \int_{\surfh} \frac{1}{\tau} \solu^{(m)} \cdot \tfuh + \sqrt{2} x_{0} \frac{\snabh \soluh^{(m)} : \snab \tfuh}{|\snabh \soluh^{(m)}|} \notag \\
& \qquad + \lambda_{coupling}^{(m)} \Big{(} \solu_{c,h} + l_{0} \frac{\soluh^{(m)} - \solu_{c,h}}{|\soluh^{(m)} - \solu_{c,h}|} \Big{)} \cdot \tfuh + \frac{\lambda_{p}}{|V_{h}(\soluh^{(m)})|} \unitnormalh \cdot \tfuh d\sigma_{h}, \label{eq:MixedBleba} \\ 
\int_{\surfh} \snabh \soluh^{(m+1)} & : \snabh \tfwh - \solwh^{(m+1)} \cdot \tfwh d\sigma_{h} = 0, \label{eq:MixedBlebb}
\end{align}
with 
\[
 \lambda_{coupling}^{(m)} = \lambda_{l} \big{(} 1 + k_{L} H(u_{R} - |\soluh^{(m)} - \solu_{c,h}|) \big{)} H (u_{B} - |\soluh^{(m)} - \solu_{c,h}|)
\]
\end{problem}

\subsection{Implementation}

We have solved the above problem using the Python bindings from the DUNE-FEM module \cite{DedKloNolOhl_Comp_2010}, which is based on the \emph{Distributed and Unified Numerics Environment} (DUNE) \cite{BastianEtAl_DUNE_2}. DUNE is an open source C++ environment that uses a static polymorphic interfaces to describe grid based numerical schemes. The package provides a large number of realisations of these interfaces including a large number of finite element spaces on structured and unstructured grids. This approach allows for the efficient and flexible simulation of a large variety of mathematical models based on partial differential equations. 

The Python bindings described in \cite{DedNol_2018_prep} simplify the rapid prototyping of new schemes and models, while maintaining the efficiency and flexibility of the DUNE framework. This is achieved by using the domain specific language UFL \cite{AlnEtAl_ACMDL_2014} to describe the mathematical model and implementing the high level program control within Python, while
carrying out all computationally critical parts of the simulation in C++ using just in time compilation of the required DUNE components. Consequently, the assembly of the bilinear forms and solving of the linear and non linear problems is implemented in C++ while the time loop and the input and output of data is carried out using the Python scripting language. 

Meshes can be provided using a GMsh file or, as done for this work, by using the internal \emph{Dune Grid Format} (DGF). All simulations reported on in this paper were performed using a first order Lagrange space over an simplicial, locally adaptive, distributed grid, which can be used for both bulk and surface domains \cite{AlkaemperEtAl_ALUGrid}. Bindings for a number of different solver packages are available through DUNE-FEM including the iterative solvers from DUNE-ISTL \cite{BlattBastian_ISTL_2007} (used for this work), direct solvers from the SuiteSparse package, and a number of solvers and preconditioners from the PetSc package. The simulation results were exported using VTK and visualised using ParaView \cite{Aya_Paraview_2015}. 

In the following we show how to setup the grid and how some parts of the mathematical model are defined within UFL. The full code needed to perform the simulations shown in this paper is available (see Data Availability Statement at the end of this paper).

The first listing shows how to read in a grid for a cell obtained from experimental data (see Section \ref{sec:imgsim} for more detail on the corresponding simulations):
\begin{lstlisting}[texcl]
  from dune.alugrid import aluConformGrid
  from dune.fem.space import lagrange
  surfaceGrid = aluConformGrid("cell.dgf", dimgrid=2, dimworld=3)
  solutionSpace = lagrange(surfaceGrid, dimRange=3, order=1, storage="istl")

  # \upshape a vector-valued finite element function for the position, 
  # \upshape initialised with the vertex positions of the initial grid 
  position  = solutionSpace.interpolate(lambda x: x, name="position")
  # \upshape another finite element function, later on used to store the previous time step
  position_n = position.copy()
\end{lstlisting}

The following snippet demonstrates how the bending terms and tension terms are defined using UFL. The remaining terms, e.g., for the pressure and the linker-molecules, are defined in a very similar way: 
\begin{lstlisting}[texcl]
  # \upshape test and trial function used to define the bilinear forms
  u   = TrialFunction(solutionSpace)
  phi = TestFunction(solutionSpace)
  w   = TrialFunction(solutionSpace)
  eta = TestFunction(solutionSpace)
  
  def invNormNxN(eta):
    S1, S2, S3 = grad(eta[0]), grad(eta[1]), grad(eta[2])
    return sum( [ S1[i]*S1[i]+S2[i]*S2[i]+S3[i]*S3[i] for i in range(3) ] )

  # \upshape the bending terms using operator splitting 
  bending_im = lam_b * inner(grad(w), grad(phi))
  op_split_pos_im = inner(grad(u), grad(eta))
  op_split_curv_im = -inner(w, eta)
  # \upshape the tension terms 
  tension_im = inner(grad(u), grad(phi))
  tension_ex = sqrt(2.0) * x_0 * 1/NormNxN(position_n) *\
               inner(grad(position_n), grad(phi))
\end{lstlisting}

In each time step a saddle point problem is solved using a Uzawa-type algorithm where a CG method is used to invert the
Schur complement as described, for example, in \cite{braess_2007}. The main algorithm is implemented in Python calling C++ routines to compute the matrix-vector operations and to solve the inner problem. The time loop with the solver is slightly to large to list here but, as stated above already, the whole code is publicly available, see the Data Availability Statement at the end of the paper for further information. 

A number of tests have been performed for problems with known solutions $(\solu,\solw)$ to validate the convergence (rates) of Theorem \ref{theo:main} and Corollary \ref{cor:errestim}. Recall that the choices of the tension term $\psi$ and the coupling term $\bbb{k}$ in the specific model \eqref{eq:PDEnondim} do not satisfy the requirements of the analysis. However, in our simulations, the denominators in these terms did not become very small. Comparative simulations with the regularised choices \eqref{eq:psip_eps} and \eqref{eq:fcp_eps} with $\eps = 10^{-5}$ did not reveal any essential difference. For conciseness, we don't report on these code validations in detail but focus on an investigation of the parameter space instead.

\subsection{Influence of the initial geometry}
\label{sec:geomsim}

The software framework allows for assessing the impact of geometries on blebbing propensity. One point of interest has been whether surface tension and pressure are sufficient to initiate blebbing without any weakening of the cortex, as found in \cite{ColEtAl_Nature_2017} in 2D. We also further study the parameter space but remark that the simulation results are at a qualitative level. An in-depth discussion involving quantitative information is beyond the scope of this article and left for future investigations. 

We consider an initial shape $\surf$ obtained by deforming a sphere of radius one by (all lengths in $\mu m$) 
\begin{equation} \label{eq:deformation}
\bbb{y} = ( y_{1}, y_{2}, y_{3} ) \to ( 4y_{1}, 4y_{2}, \ttt{y}_{3} ), \quad \ttt{y}_{3} = \sign(y_{3}) \begin{cases} (3 - \cos(\pi r/2))/2, & \quad \mbox{if } r \leq 2, \\ \sqrt{4 - (r-2)^2}, & \quad \mbox{if } r > 2, \end{cases}
\end{equation}
with $r = \sqrt{(4y_{1})^2 + (4y_{2})^2}$. This yields a shape similar to a discocyte (or red blood cell, see Figure \ref{fig:01_mesh}) with a volume of about $V(\bbb{id}_{\surf}) \approx 150 \mu m^3$ and a largest distance of $4 \mu m$ from the centre. Parameters for the various forces vary in the literature, not least due to differing cell types and differences in the models. For the tension coefficient we chose $k_{p} = 15 pN/ \mu m$ (ranges from $2 pN / \mu m$ \cite{LimKoonChiam_ComputMethBiomechBiomedEng_2013} to $100 pN/ \mu m$ \cite{ManYanLowAll_BiohysJ_2016}), for the bending coefficient $k_{b} = 0.075 pN \mu m$ (between $0.01 pN \mu m$ \cite{WooEtAl_Biomech_2014} and $0.2 pN \mu m$ \cite{LimKoonChiam_ComputMethBiomechBiomedEng_2013}), and for the linker spring coefficient $k_{l} = 270 pN / \mu m^3$ (close to $267 pN /\mu m^3$ in \cite{StrGuy_MathMedBiol_2013}). The parameters $x_{0} = 0.95$, $l_{0} = 40 nm$, and $u_{B}= 56 nm$ were chosen as in \cite{ColEtAl_Nature_2017}. The parameters $u_{R} = 7.5nm$ and $k_{L} = 500.0$ were chosen ad hoc but repeating some simulations with $k_{L} = 0$ (particularly those with higher tension so that the membrane got closer to the cortex) didn't reveal any visual difference. The pressure difference $p_{0}/|V(\solu(0))| \approx 2.25 Pa$ turned out sufficient to initiate blebbing without cortex weakening. This is smaller than values found in the literature (between $10Pa$ \cite{WooEtAl_Biomech_2014} and $81 Pa$ \cite{TysZatBreKay_PNAS_2014}) but we note that the dimension is higher and the model does not account for the stiff cortex. With a length scale of $U = 1\mu m$ the set of non-dimensional parameters is stated in Table \ref{tab:01_geomparam} and was used for simulations unless stated otherwise. 

A triangulation $\surfh$ is obtained by starting with a cube with vertices on the unit-sphere, then diagonally cutting the square faces into triangles, and then bisecting all triangles 14 times such that the longest edge is halved including projecting the new vertices to the unit-sphere after each refinement step. After, the above map \eqref{eq:deformation} is applied to the 196608 vertices. Figure \ref{fig:01_mesh} gives an impression of a mesh thus obtained but with a ten refinements only. The time step size was set to $\tau = 0.0025$ and time stepping ended at $T=2$. At that end time the final shapes usually weren't at rest yet but the deformations were sufficient to draw qualitative conclusions. 

Figure \ref{fig:02_BSKT} gives an overview of some shapes at the final time for the data set in Table \ref{tab:01_geomparam} and some variants (see Figure caption for details). Axisymmetry of the initial shape seems preserved, which suggests comparing cuts through the centres for more insight. Figure \ref{fig:03_BSKT1} displays the slices through the initial and the final shape that is visible in Figure \ref{fig:02_BSKT}A. The color code from Figure \ref{fig:02_BSKT} is used again so that parts of the membrane with broken linkers are coloured red. Differences are predominant in the concave part of the initial shape, where the membrane has moved outwards and detached. The tension force in such concave parts points outwards and, together with the pressure, initiates a bleb without requiring any weakening of the cortex. This simulation thus supports the finding in \cite{ColEtAl_Nature_2017}. 

In Figure \ref{fig:04_BSKT2} we compare the final shapes for different parameters of the linker strength $\lambda_{l}$, more precisely, slices of the shapes in Figures \ref{fig:02_BSKT}A and \ref{fig:02_BSKT}B. Note that the color code is different (see caption of Figure \ref{fig:04_BSKT2}). The deformation isn't much stronger as, once the membrane is detached, the linker terms doesn't influence the evolution any further. But a weaker linker strength $\lambda_{l}$ and, thus, less resistance to breaking leads to a wider bleb site. 

A smaller resting length parameter $x_{0}$ increases the surface tension, which leads to a faster evolution and a stronger final deformation. This is visible in Figure \ref{fig:05_BSKT3} where we compare the slices of the shapes in Figures \ref{fig:02_BSKT}A and \ref{fig:02_BSKT}C, and the (red) curve for the smaller $x_{0}$ indicates that the membrane has moved further away from the initial shape. 

The impact of a higher pressure is illustrated in Figure \ref{fig:06_BSKT4} where slices through the shapes in Figures \ref{fig:02_BSKT}A (blue) and \ref{fig:02_BSKT}D (red) are overlayed. The effect resembles a bit that of a smaller linker strength in that the deformation isn't much different and in that the bleb site is much bigger. The pressure term doesn't break down after detachment and continues to push outwards, though, so that the membrane has moved a bit further throughout the bleb site.

\subsection{Application to experimental data}
\label{sec:imgsim}

Apart from given, 'in-vitro' geometries and their influence on blebbing, users may also be interested in studying the effect of 'in-vivo' geometries that are obtained from experimental data. The image postprocessing outlined in \cite{DuHawSteBre_BMCBioinf_2013} enables to extract triangulated surfaces representing the cell membrane from 3D images of cells, which then can be steered into the software framework. This was done with data of a Dictyostelium cell (also used in \cite{DuHawSteBre_BMCBioinf_2013}) moving by actin-driven pseudo-pods without any blebbing. However, the purpose is again to showcase the capability of the software framework rather than to extract any quantitative information, which is left for future investigations. 

We used the non-dimensional parameters in Table \ref{tab:02_imgparam} and with $T = 20$ and $\tau = 0.02$. Figure \ref{fig:07_CD}, left, shows the triangulated surface $\surfh$ that has been obtained from the image data. On the right of Figure \ref{fig:07_CD} the final shape is displayed where the same colour code as in Figure \ref{fig:02_BSKT} for the deformation strength is used. As in the simulations before we observe that blebs form in concave regions. We also see some deformations at the sides where small protrusions become quite spiky. Both tension and resistance to bending are expected to prevent any singularities to occur, however the geometry seems under-resolved by the mesh in these areas.

%%%%%%%%%%%%%%%%%%%%%%%%%%%%%%%%%%%%%%%%%%%%%%%%%%%%%%%%%%%%%%%
\section{Conclusion}

A general modelling framework for the onset of blebbing has been presented and analysed. It is formulated in terms of partial differential equations on the initial membrane, which is considered as a hypersurface. Various forces acting on the plasma membrane due to its elastic properties, linker molecules coupling it to the cell cortex, and cell internal pressure are accounted for. Fluid flow within and outside of the cell is essentially neglected modulo a drag force but may be considered in future studies. 

The general framework is particularly flexible with regards to membrane tension and the coupling forces. A convergence analysis of a surface finite element discretisation shows its robustness to model alterations within not too restrictive limits. There are some open questions with regards to the discretisation in time, and as blebs are local events, spatial mesh adaptivity may be beneficial. 

Software for a specific instance of the general model is provided and has been used to perform some numerical simulations. A convenient high-level interface in Python allows for directly implementing the model in its variational form and solving it by an efficient software backend. Standard software usually does not provide functionality for numerically solving problems on moving domains or hypersurfaces in 3D out of the box but requires a substantial amount of coding. We hope that our approach will address this issue and simplify the implementation of such moving boundary problems.

\section*{Conflict of Interest Statement}

The authors declare that the research was conducted in the absence of any commercial or financial relationships that could be construed as a potential conflict of interest.

\section*{Author Contributions}

Bj\"{o}n Stinner provided context and background, significantly contributed to the model and the numerical analysis, and set most of the paper. \\
Andreas Dedner was core developer of the Python bindings and of the DUNE
software framework and set parts of the paper. \\
Adam Nixon contributed to the model and the numerical analysis, performed the simulations, and set parts of the paper. 

\section*{Funding}

This project was supported by the Engineering and Physical Sciences Research Council (EPSRC, United Kingdom), grant numbers EP/K032208/1 and EP/H023364/1. 

\section*{Acknowledgments}

The authors would like to thank the Isaac Newton Institute for Mathematical Sciences, Cambridge, for support and hospitality during the programme \emph{Geometry, compatibility and structure preservation in computational differential equations}, where work on this paper was undertaken. 

\section*{Data Availability Statement}

The Python scripts used to obtain the results reported on here are available in a git repository hosted on the DUNE gitlab server: \\
\url{https://gitlab.dune-project.org/bjorn.stinner/sfem_blebs}.

The two main scripts are {\tt blebbing\_artgeom.py} and {\tt blebbing\_imgdata.py} used for the results from Sections
\ref{sec:geomsim} and \ref{sec:imgsim}, respectively. The setup of the model, the time loop, and the solver used in both main
scripts are contained in {\tt blebbing\_compute.py}. Some auxiliary functions can be found in {\tt blebbing\_tools.py}.

Both scripts can be executed using the DUNE-FEM docker container. A script `startdune.sh` is available in the git repository to download and start the container. This requires the 'docker' software to be available on the system. It can be downloaded for different platforms including Linux, MacOS, and the latest Windows version. More information is available under \\
\url{https://dune-project.org/sphinx/content/sphinx/dune-fem/installation.html}. 

\bibliographystyle{plain}
\bibliography{sfem_blebs}

\section*{Figures and tables}

\begin{table}[h]
\begin{center}
\begin{tabular}{|l|l|l|l|l|l|l|l|}
\hline 
\rule{0pt}{12pt} $x_0$ & $\lambda_{b}$ & $\lambda_{l}$ & $l_{0}$ & $u_{B}$ & $k_{L}$ & $u_{R}$ & $\lambda_{p}$ \\ \hline 
\rule{0pt}{12pt} 0.95  &         0.005 &            18 &    0.04 &   0.056 &   500.0 &    0.0075 &          22.5 \\ \hline 
\end{tabular}
\end{center}
\caption{Standard non-dimensional parameters for numerical simulations with a given geometry, see Section \ref{sec:geomsim} for further details.}
\label{tab:01_geomparam} 
\end{table}

\begin{table}[h]
\begin{center}
\begin{tabular}{|l|l|l|l|l|l|l|l|}
\hline 
\rule{0pt}{12pt} $x_0$ & $\lambda_{b}$ & $\lambda_{l}$ & $l_{0}$ & $u_{B}$ & $k_{L}$ & $u_{R}$ & $\lambda_{p}$ \\ \hline 
\rule{0pt}{12pt} 0.95  &         0.125 &          0.72 &     0.2 &    0.28 &   500.0 &    0.15 &         150.0 \\ \hline 
\end{tabular}
\end{center}
\caption{Non-dimensional parameters for numerical simulations with an initial surface obtained from image data, see Section {sec:imgsim} for further details.}
\label{tab:02_imgparam} 
\end{table}

\begin{figure}[h!]
\begin{center}
 \includegraphics[width=6cm]{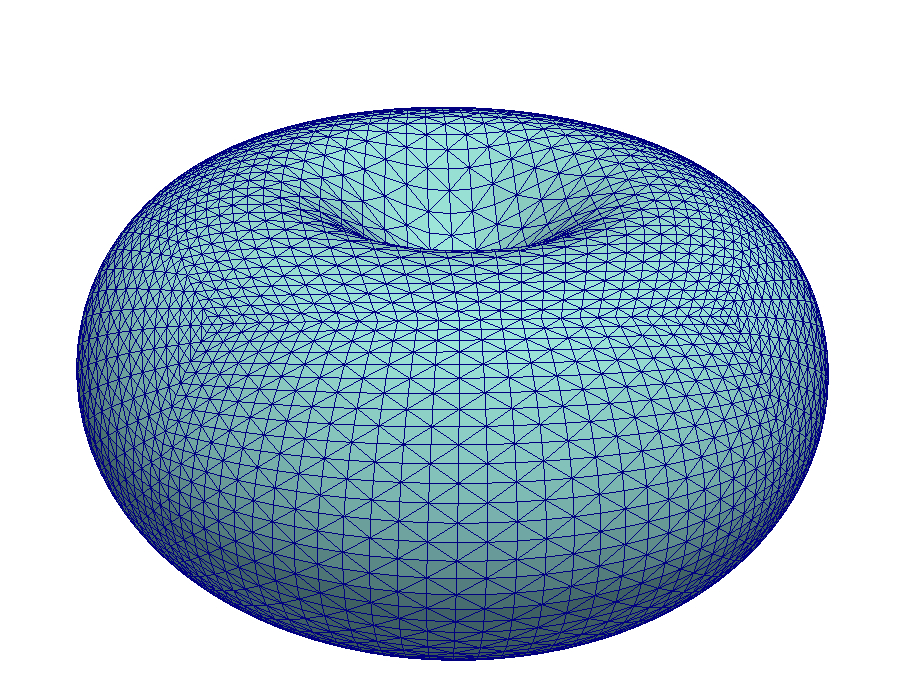}
\end{center}
\caption{Illustration of the shape used in Section \ref{sec:geomsim} and a mesh $\surfh$. For better visibility of the triangles, only ten bisections were performed resulting in a mesh with 20480 vertices. A finer mesh with 196608 vertices was used for the computations.} 
\label{fig:01_mesh} 
\end{figure}

\begin{figure}[h!]
\begin{center}
\includegraphics[width=16cm]{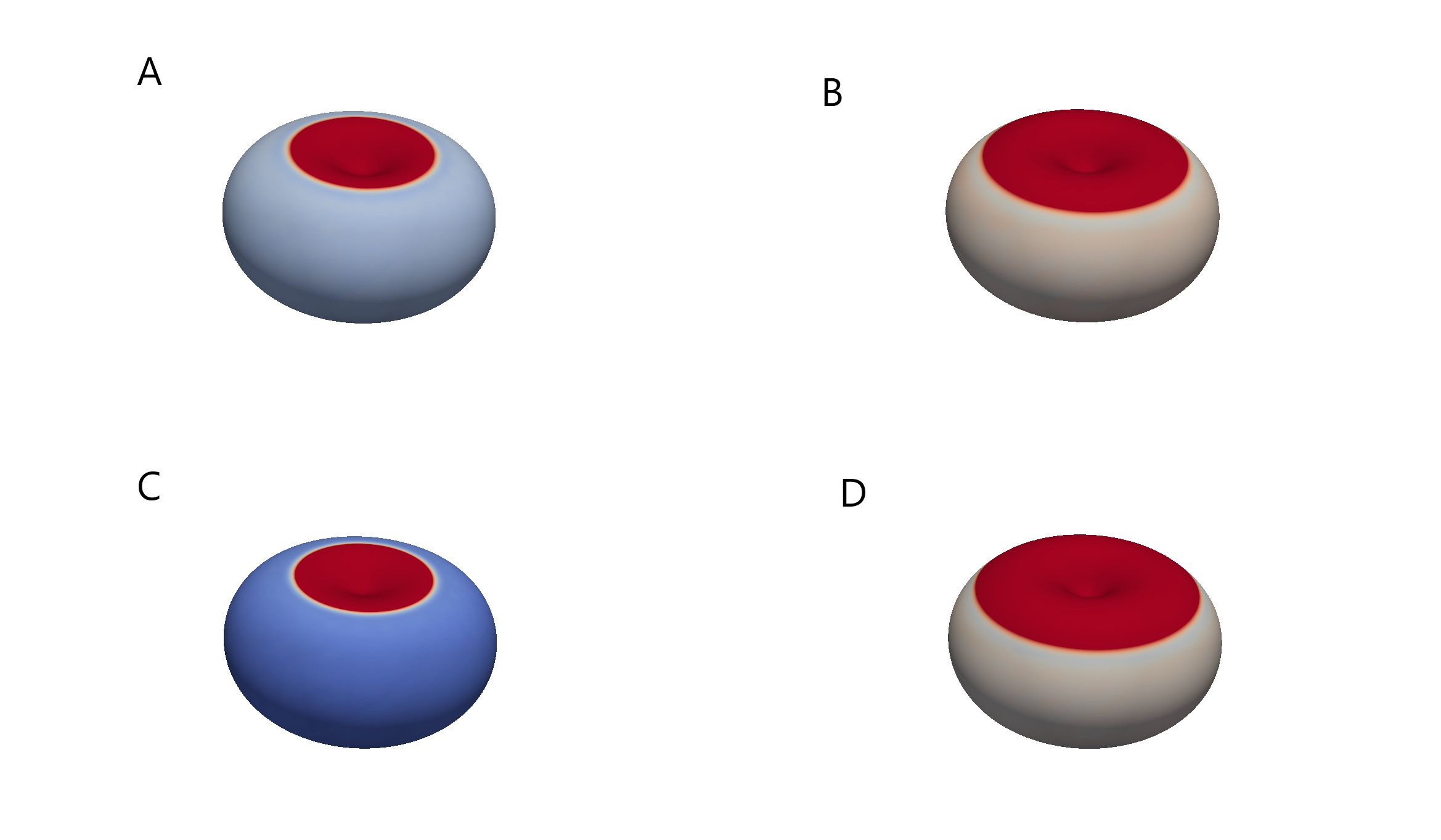}
\includegraphics[width=7cm]{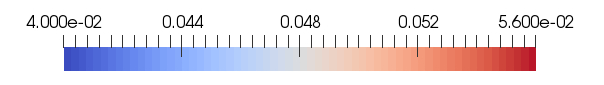}
\end{center}
\caption{Final shapes for computations with the initial shape in Figure \ref{fig:01_mesh}. The colour scheme indicates the distance of the membrane to the cortex $|\soluh - \solu_{h,c}|$. Values below the resting length $l_0 = 0.04$ are highlighted in blue and values above the critical length of breaking $u_{B} = 0.056$ in red, whilst values in between are shaded as indicated on the bar. The parameters in Table \ref{tab:01_geomparam} lead to the upper left shape (A). For B, the linker strength was reduced by setting $\lambda_{l} = 12$. For C, the tension was increased by setting $x_{0} = 0.85$. For $D$, the pressure was increased by setting $p_{0} = 30$.}
\label{fig:02_BSKT} 
\end{figure}

\begin{figure}
\begin{center}
\includegraphics[width=16cm]{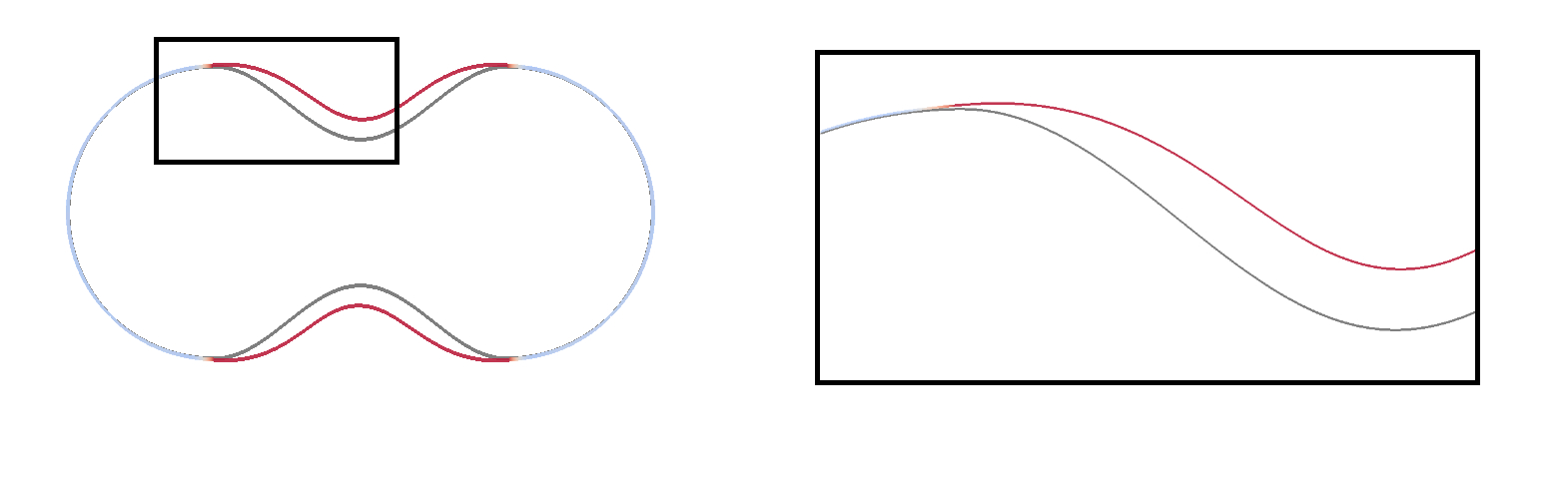}
\end{center}
\caption{Slice of the initial and final shape (latter on top). Simulation data in Table \ref{tab:01_geomparam}, see Section \ref{sec:geomsim} for further simulation details. A magnified image of the black box is presented on the right. The color code is as in Figure \ref{fig:02_BSKT}.}
\label{fig:03_BSKT1}	
\end{figure}

\begin{figure}
\begin{center}
\includegraphics[width=16cm]{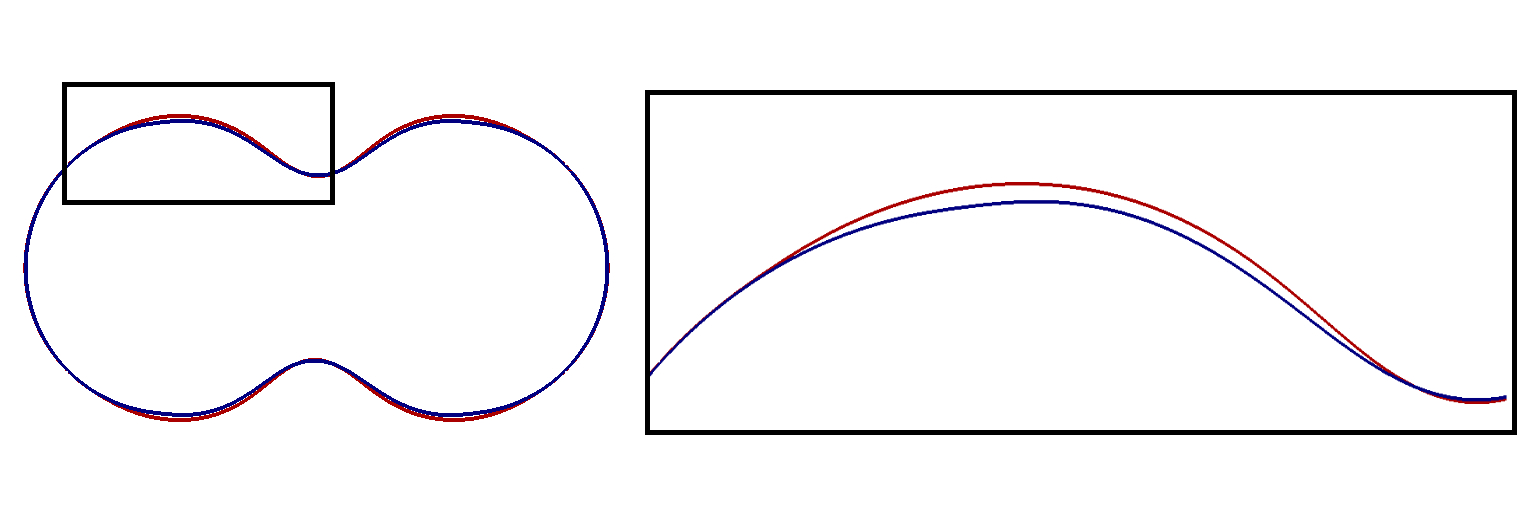}
\end{center}
\caption{Slices of final shapes for parameters as in Table \ref{tab:01_geomparam} but different linker strengths, namely $\lambda_{l} = 18$ (blue) and $\lambda_{l} = 12$ (red), with the latter on top. A magnified image of the black box is presented on the right. See Section \ref{sec:geomsim} for further details.}
\label{fig:04_BSKT2}
\end{figure}

\begin{figure}
\begin{center}
\includegraphics[width=16cm]{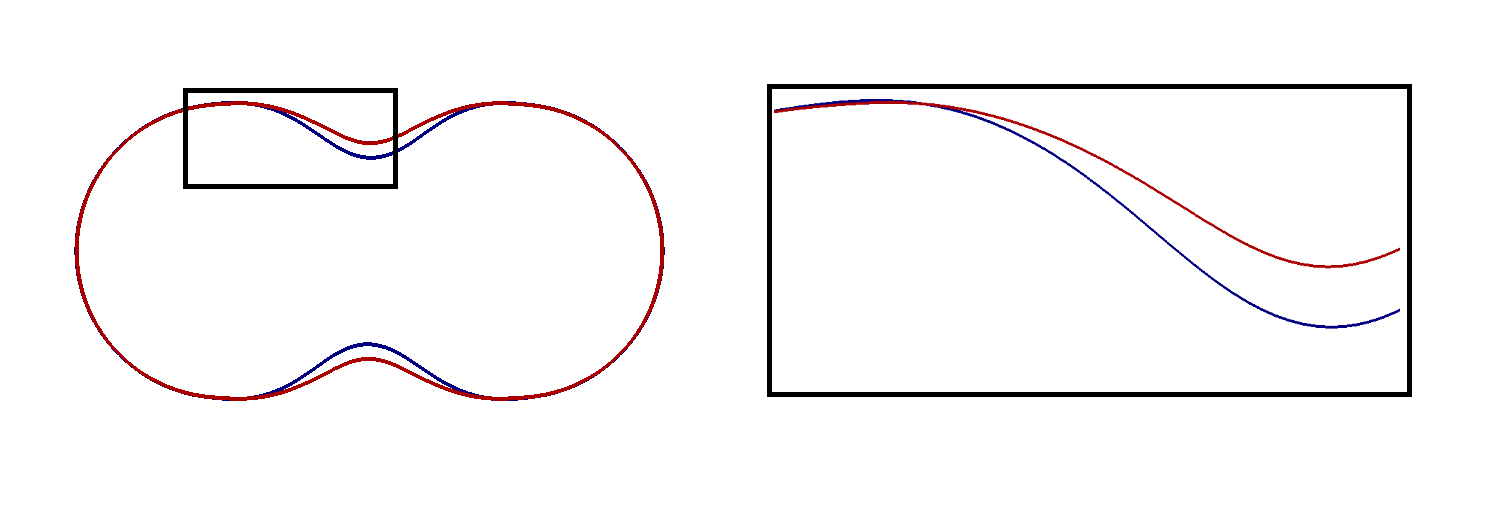}
\end{center}
\caption{Slices of final shapes for parameters as in Table \ref{tab:01_geomparam} but different membrane tensions, we chose $x_{0} = 0.95$ (blue) and $x_{0} = 0.85$ (red), with the latter on top. A magnified image of the black box is presented on the right. See Section \ref{sec:geomsim} for further details.}
\label{fig:05_BSKT3}
\end{figure}

\begin{figure}
\begin{center}
\includegraphics[width=16cm]{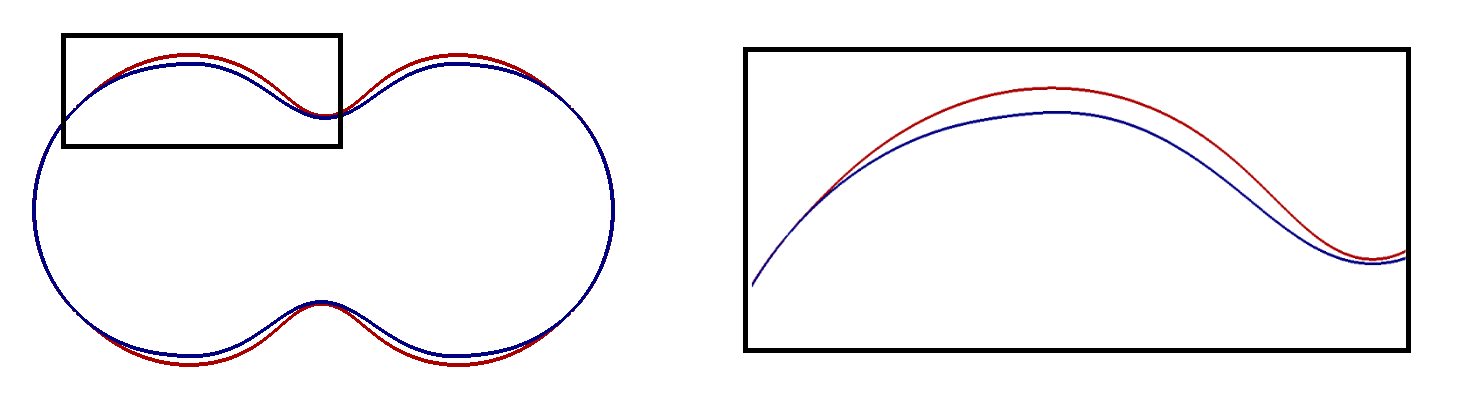}
\end{center}
\caption{Slices of final shapes for parameters as in Table \ref{tab:01_geomparam} but different pressure parameters, we set $p_{0} = 22.5$ (blue) and $p_{0} = 30$ (red), with the latter on top. A magnified image of the black box is presented on the right. See Section \ref{sec:geomsim} for further details.}
\label{fig:06_BSKT4}
\end{figure}

\begin{figure}
\begin{center}
\includegraphics[width=7cm]{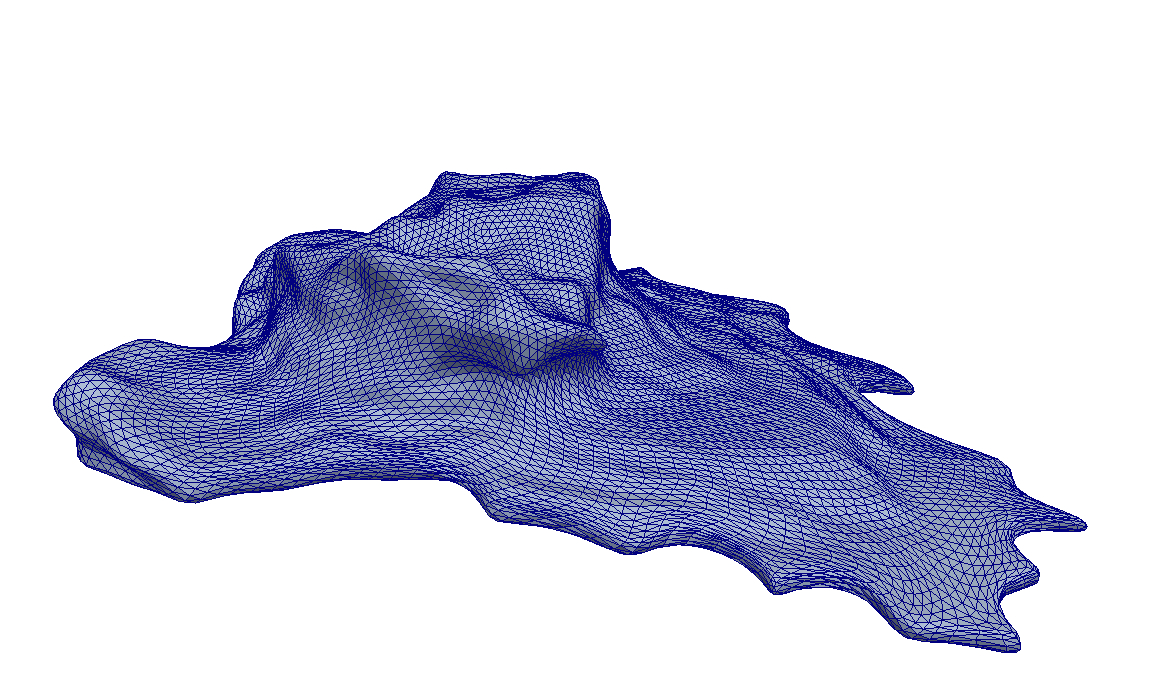} \hfill \includegraphics[width=7cm]{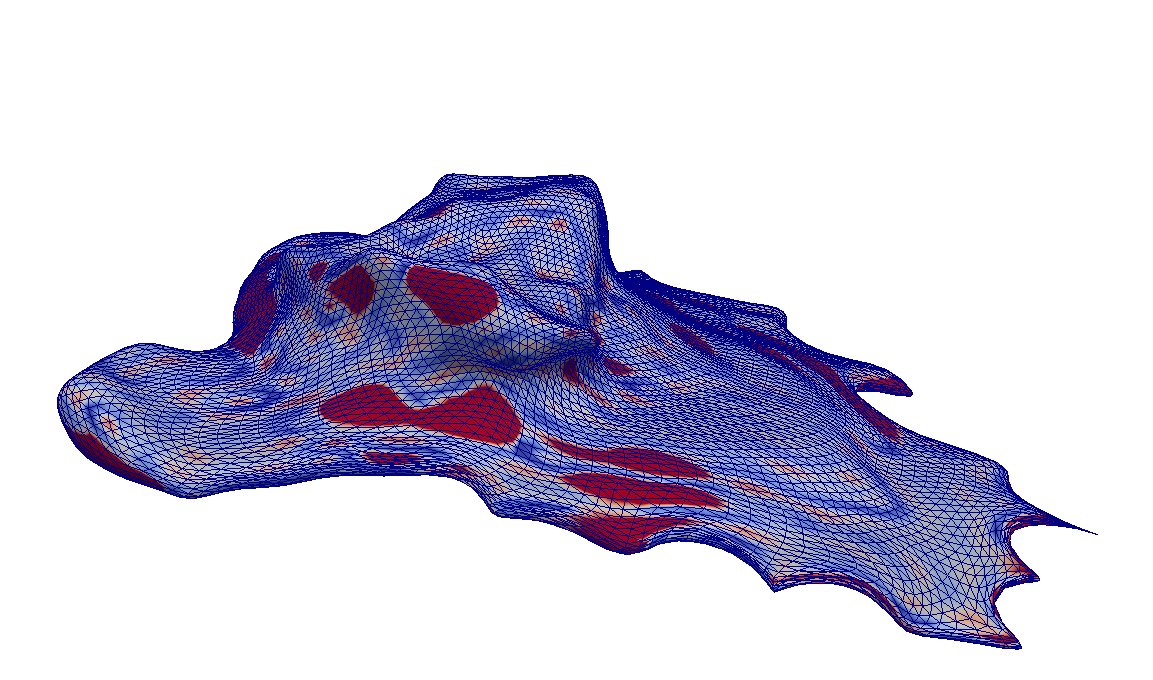}
\end{center}
\caption{Application of the scheme in Problem \ref{prob:discrete} to a cell surface obtained from image data. The color scheme is as in Figure \ref{fig:02_BSKT}. See Section \ref{sec:imgsim} for further details .}
\label{fig:07_CD}
\end{figure}

\end{document}